\documentclass[11pt,a4paper]{article}
\usepackage{amssymb, amsmath}
\usepackage{graphicx} % Required for inserting images
\usepackage{subfigure}
\usepackage{float}  
\usepackage{amsfonts}
\usepackage{amsthm}
\usepackage{xcolor}
\usepackage[bookmarks=true,hidelinks]{hyperref}
\usepackage[normalem]{ulem}
\usepackage{enumitem}
\usepackage[font=small]{caption}

\usepackage{comment}

\DeclareMathOperator*{\argmin}{arg\,min}

\newtheorem{theorem}{Theorem}

\newtheorem{corollary}[theorem]{Corollary}
\newtheorem{lemma}[theorem]{Lemma}
\newtheorem{definition}[theorem]{Definition}
\newtheorem{remark}[theorem]{Remark}

\numberwithin{equation}{section}
\numberwithin{theorem}{section}

\renewcommand{\AA}{\mathsf{A}^{(i)}}
\newcommand{\WW}{\mathsf{W}}

\newcommand{\ww}{\mathsf{w}^{(i)}}
\newcommand{\bb}{\mathsf{b}^{(i)}}
\newcommand{\R}{{\mathbb R}}
\newcommand{\Gspace}{\mathcal{G}}
\renewcommand{\leq}{\leqslant}
\renewcommand{\geq}{\geqslant}
\newcommand{\eps}{\varepsilon}
\renewcommand{\aa}{\mathsf{a}^{(i)}}
\newcommand{\dd}{\mathsf{d}}
\newcommand{\ee}{\mathsf{e}}
\newcommand{\xx}{\mathsf{x}}

\renewcommand{\tt}{\tilde{t}}
\renewcommand{\ss}{\mathsf{s}}
\newcommand{\sgn}{\text{sgn}\,}

\newcommand{\nullspace}[1]{{\mathcal{N}(#1)}}

\title{Estimating neural connection strengths from firing intervals}
\author{Maren B. Kristoffersen$^*$, Bjørn F. Nielsen$^*$, and Susanne Solem\thanks{Department of Mathematics, Norwegian University of Life Sciences, NO-1433 Ås, Norway (corresponding e-mail: maren.brathen.kristoffersen@nmbu.no).}}

\date{}

\begin{document}

\maketitle

\begin{abstract}
    We propose and analyse a procedure for using a standard activity-based neuron network model and firing data to compute the effective connection strengths between neurons in a network. We assume a Heaviside response function, that the external inputs are given and that the initial state of the neural activity is known. The associated forward operator for this problem, which maps given connection strengths to the time intervals of firing, is highly nonlinear. Nevertheless, it turns out that the inverse problem of determining the connection strengths can be solved in a rather transparent manner, only employing standard mathematical tools. In fact, it is sufficient to solve a system of decoupled ODEs, which yields a linear system of algebraic equations for determining the connection strengths. The nature of the inverse problem is investigated by studying some mathematical properties of the aforementioned linear system and by a series of numerical experiments. Finally, under an assumption preventing the effective contribution of the network to each neuron from staying at zero, we prove that the involved forward operator is continuous. Sufficient criteria on the external input ensuring that the needed assumption holds are also provided.  
\end{abstract}


\section{Introduction}
Motivated by modern calcium imaging techniques which can supply accurate firing data for a large number of neurons, see for example \cite{Zongetal2022}, this manuscript concerns a method for the inverse problem of estimating the effective connection strengths between neurons in a network described by a standard activity-based neuron network model of $n$ neurons. The model, which builds on the classical works by Wilson and Cowan \cite{WilsonCowan1972, WilsonCowan1973} and Amari \cite{Amari1977}, reads as follows: for $i=1, \, 2, \, \ldots, \, n$, 
\begin{equation} \label{eq:motherODE}
\begin{split}
\tau_{i} \frac{d s_i}{dt}(t) + s_i(t) &= \Phi\left( \sum_j W_{ij} s_j(t-\tau_d) + B_i (t) \right), \quad \, 0 < t \leq T, \\
s_i(t) &= s_i^0 e^{-t/\tau_i}, \quad -\tau_d < t \leq 0, 
\end{split}
\end{equation}
where $s_i$ represents the synaptic drive of neuron $i$ with initial state $s_i^0 \geq 0$ and time constant $\tau_i$. The expression on the right-hand side represents the firing rate of the neuron \cite{Bressloff2012, Ermentrout2010}. The modulation function $\Phi$ adjusts the input received by neuron $i$, $W_{ij} \in \mathbb{R}$ represents the effective connection strength from neuron $j$ to neuron $i$, $B_i \in L^\infty([0,T])$ is the external input to the $i$'th neuron, $\tau_d > 0$ incorporates a time delay, and $[0,T]$ is the time interval of interest.

We will assume that the modulation function $\Phi$ is the Heaviside function. Utilising \eqref{eq:motherODE} with Heaviside modulation is done for mainly two reasons. First, although a coarse simplification of a real neuronal network, it is classical and variants of it appears in a plethora of works in neuroscience and related fields where the aim is to better understand various aspects of our brain. It is therefore highly relevant to investigate whether obtaining quantitative information from this model is possible. Second, surprisingly given the nonlinearity of \eqref{eq:motherODE}, the inverse problem of estimating the connection strengths (or ``connectivities'') from firing data turns out to be quite manageable. As an additional motivation, successfully obtaining information about the connection strengths between neurons based on the present framework can provide valuable insight into their wiring when performing various tasks.

Our firing data will be the \emph{firing intervals} of each neuron. In the context of \eqref{eq:motherODE}, we define the firing intervals to be the time intervals for which the right-hand side, the firing rate, is equal to 1. Note that, as $\Phi$ is assumed to be the Heaviside function, the right-hand side is either 0 or 1 at any time. Given the firing intervals, the procedure for the inverse problem then only consists of solving a system of decoupled ODEs and subsequently a linear system of algebraic equations. 

Calcium imaging techniques are able to record over 1000 neurons simultaneously \cite{Zongetal2022}. What is observed in these recordings is the intracellular Ca$^{2+}$ concentration, which can be used as a measure of the activity of a neuron. There is a connection between the Ca$^{2+}$ concentration of a neuron and its firing rate \cite{deneux2016accurate}, and it can therefore be argued that the output of calcium imaging recordings can be used to determine when the Heaviside firing rate in \eqref{eq:motherODE} should be 0 or 1. Our work is, however, a first investigation in this direction, and the the aim here is to analyse in what sense it is possible to solve the present inverse problem. Therefore we will only focus on recovering the connection strengths, and assume throughout that $T$, $\{ s_i^0 \}_{i=1}^n$, and $\tau_d$ are given quantities, that $\{ B_i \}_{i=1}^n$ are given functions, and that the time constants are identical and equal to one: $\tau_i = 1$ for all $i$. All these parameters are not available in real calcium imaging experiments, and thus future investigations are needed in order to fully explore the applicability of this model.

Closest to the present work are the investigations by Potthast and beim Graben \cite{PotthastGraben2009-1, PotthastGraben2009-2} without delay, and \cite{alswaihli2018kernel} which builds on the first two, incorporating delay. These works also focused on the inverse problem of estimating the connectivities of a network of neurons. However, they differ from the present in several aspects: i) the voltage version of \eqref{eq:motherODE} is utilised, meaning that the modulation function is inside the sum, ii) a smooth modulation function is assumed, and iii) the procedure relies on knowledge about the solution and its time derivative at every point in time.

We are mainly concerned with the inverse operation of computing information about the connectivities from the knowledge of the time intervals of firing. However, we have dedicated a section to investigating the continuity properties of the forward operator mapping connectivities to firing intervals. That the forward operator is continuous is crucial in order to reliably compute the time intervals of firing with finite precision arithmetic as round-off errors could potentially have a devastating effect on the outcome of a simulation \cite{Nielsen2016}. We show that the forward operator is (locally) continuous under a threshold condition assuming that the effective input to each neuron, i.e., the argument of $\Phi$ in \eqref{eq:motherODE}, does not stay at zero for any period of time. The assumption is related to previously established conditions needed for the existence and uniqueness of solutions to the discrete or continuous in space voltage version of \eqref{eq:motherODE} with a Heaviside modulation function and zero time delay, $\tau_d=0$, see \cite{potthast2010existence, Kruger2018, Nielsen2016, Nielsen2017}. Note that when $\tau_d>0$, the question of existence and uniqueness is not an issue for \eqref{eq:motherODE} with a Heaviside modulation function under reasonable assumptions on the external inputs. Thus, we require $\tau_d$ to be strictly positive.

One drawback with the threshold conditions mentioned above is that they depend on the unknown solution. In \cite{Kruger2018} sufficient conditions on the initial data are given in order to ensure that their threshold condition holds (called the \emph{absolute continuity condition}) in the case of the spatially continuous voltage version of \eqref{eq:motherODE}. In particular, their condition holds if the spatial gradient of the initial data is non-zero at the threshold value of the Heaviside function and that the initial data crosses the threshold a finite number of times. This is a restriction on the initial data that we cannot guarantee if the present investigations are to be relevant when considering data from real experiments. We instead identify sufficient conditions on the external inputs $\{B_i\}$, which one to some extent can control, ensuring that our threshold assumption holds.

The manuscript is organised as follows. Section \ref{sec:prelim} contains useful observations needed throughout the manuscript, while in Section \ref{sec:method} we investigate the nature of the inverse problem of determining the connectivities from the firing intervals and present a method. Section \ref{subsec:continuity_forward} is dedicated to establishing (local) continuity of the forward operator, Section \ref{sec:num} is devoted to numerical experiments, and we end the manuscript with some concluding remarks concerning the established results in Section \ref{sec:concluding}.

\section{Preliminaries}\label{sec:prelim}
For the model \eqref{eq:motherODE}, we define firing as follows. 
\begin{definition}[Firing]\label{def:firing}
The $i$'th neuron is said to fire when the right-hand side of the ODE in \eqref{eq:motherODE} equals one, i.e., when $\Phi$'s argument is larger than or equal to zero\footnote{For the sake of convenience, we define $$\Phi(x) = \left\{ \begin{array}{cc}
    0 & x<0,  \\
    1 & x \geq 0,
\end{array} \right.$$ which differs slightly from the standard convention that the Heaviside function equals $0.5$ for $x=0$.}. Thus, a neuron in this context is said to fire when its firing rate is equal to one.
\end{definition}
\noindent Throughout this manuscript, $\Gamma$ is the set containing the indexes corresponding to neurons that fire,  
\begin{equation} \label{set:neurons_that_fire}
    \Gamma = \{i \; | \; \textnormal{neuron } i \textnormal{ fires during the time interval } (0, T]  \},    
\end{equation}
and the $k$'th time interval of firing for the $i$'th neuron is denoted  
\begin{equation} \label{eq:firing_interval}
I_i^k= [t_i^k, \, t_i^k + \Delta t_i^k], \quad k=1,\, 2, \, \ldots, \, K(i).    
\end{equation}
Here, $K(i)$ is the number of times the $i$'th neuron fires during the time interval $(0,T]$, and $0 < \Delta t_i^k$ is the length of the $k$'th time interval of firing. Note that according to Definition \ref{def:firing},
\begin{equation*}
   \Phi\left( \sum_j W_{ij} s_j(t-\tau_d) + B_i (t) \right) = 1, \quad  t \in I_i^k.   
\end{equation*}
Furthermore, if $\Delta t_i^k = 0$, the right-hand side of the ODE in \eqref{eq:motherODE} only differs from zero at the time instance $t_i^k$, which is of measure zero and thus will not have any impact on the solution of \eqref{eq:motherODE}. 

We assume that $I_i^1, \, I_i^2, \, \ldots, I_i^{K(i)}$ are disjoint intervals and that they are ordered chronically: 
\begin{equation} \label{ineq:chronically}
   0 < t_i^1 < t_i^2 < \ldots < t_i^{K(i)}.  
\end{equation}
The time intervals of firing for different neurons are, of course, allowed to partially or fully overlap. 

Let $\WW= [W_{ij}]$ denote the matrix of connectivities. The forward operator, $F$, in the present context maps given connectivities to the time intervals of firing, 
\begin{equation} \label{def:forwardOperator}
\begin{aligned}
    F: \R^{n \times n} \quad &\to \quad \mathcal{I} , \\
    \WW \quad &\mapsto \quad  \{ I_i^k \},
    \end{aligned}
\end{equation}
which is a highly nonlinear operator due to the presence of the Heaviside response function $\Phi$ in the right-hand side of our network model \eqref{eq:motherODE}.

Let $\mathcal{M}$ be the collection of all finite unions of disjoint closed intervals $I \subset [0,T]$. Define  
\begin{equation*}
    \mathcal{I} := \mathcal{M}\times \mathcal{M} \times \cdots \times \mathcal{M} = \mathcal{M}^n,
\end{equation*}
where $n$ is the number of neurons in \eqref{eq:motherODE}.
Then, under the above assumptions, any collection of firing intervals $\{I_i^k\}:=\big\{\{I_i^k\}_{k=1}^{K(i)}\big\}_{i=1}^n$ can be uniquely identified with an element in $\mathcal{I}$. Together with the symmetric difference of sets defined below, we use this observation to define a metric allowing us to measure the difference between two different collections of firing intervals. The resulting metric will be utilised when proving the continuity of the forward operator \eqref{def:forwardOperator} in Section \ref{subsec:continuity_forward}.

\begin{definition}[Symmetric difference \cite{halmos1960naive}]
    The symmetric difference between any pair of measurable sets $A,\tilde A \subset [0,T]$ is defined by 
    \begin{equation}\label{def:general_sym_diff}
        d_\triangle(A,\tilde A) = |A\triangle \tilde A| = |(A\setminus\tilde A) \cup (\tilde A \setminus A)|.  
    \end{equation}
    If two sets $A,\tilde A$ are defined to be equivalent when they only differ on a set of measure zero, \eqref{def:general_sym_diff} defines a metric on the collection of all measurable subsets of $[0,T]$, called the symmetric difference metric. 
\end{definition}
\noindent A metric between two elements $\{I_i^k\}$, $\{\tilde{I}_i^k\} \in \mathcal{I}$, can then be defined as
\begin{equation}\label{def:symdiff_metric}
    d(\{I_i^k\},\{\tilde{I}_i^k\}) = \sup_i d_{\triangle}\big(\bigcup_k I_i^k ,\bigcup_k\tilde{I}_i^k\big).
\end{equation}
It is straightforward to check that \eqref{def:symdiff_metric} indeed defines a metric on $\mathcal{I}$. Note that two collections of intervals are deemed identical with this metric even if they differ on a set of measure zero. As mentioned above, this is okay since intervals of measure zero in this context will not impact the solution of 
\eqref{eq:motherODE_fire}.

\section{Derivation of the method}\label{sec:method}
We start this section by making some useful observations before presenting the algorithm in Section \ref{subsec:algorithm}. 

When the $i$'th neuron fires, the right-hand side of the ODE in \eqref{eq:motherODE} equals one.  
Hence, if the firing intervals $\{ I_i^k \}$ are known, we can write \eqref{eq:motherODE} in the form 
\begin{equation} \label{eq:motherODE_fire}
\begin{split}
\frac{d s_i}{dt}(t) + s_i(t) &= \sum_k \mathcal{X}_{I_i^k}(t), \quad \, 0 < t \leq T, \\
s_i(t) &= s_i^0 e^{-t}, \quad -\tau_d < t \leq0,
\end{split}
\end{equation}
for $i=1, \, 2, \, \ldots, \, n$, where $\mathcal{X}_{I_i^k}$ denotes the characteristic function of $I_i^k \subset \mathbb{R}$. We observe that \eqref{eq:motherODE_fire} is a decoupled system of $n$ inhomogeneous linear ODEs which can be solved explicitly. Note that this also implies that given the firing intervals, the $s_i$'s are independent of the connection strengths $\WW$. These considerations lead to the following lemma. 

\begin{lemma} \label{th:mother_theorem}
Assume that $\{ I_i^k \}$ is in the range of the operator $F$. Then, any solution $\WW=[W_{ij}]$ of
\begin{equation*}
    F \left( \WW \right) = \{ I_i^k \}
\end{equation*}
must satisfy, for any $i \in \Gamma$ and for any $t \in \left\{t_i^k, \, t_i^k+\Delta t_i^k\right\}_{k=1}^{K(i)}$, the linear equation 
\begin{align} \label{eq:mother_linear_system}
    \sum_j & s_j(t-\tau_d) W_{ij}   = - B_i (t),
\end{align}
whenever $B_i$ is continuous at $t$, and the inequalities 
\begin{equation} \label{eq:mother_inequalities}
\begin{split}
    \sum_j s_j(t-\tau_d) W_{ij}  &\geq - B_i (t), \quad t \in (t_i^k, \, t_i^k+\Delta t_i^k), \, k=1, \, 2, \, \ldots, \, K(i),\\
    \sum_j s_j(t-\tau_d) W_{ij}  &< - B_i (t), \quad t \notin \bigcup_{k=1}^{K(i)} I_i^k, 
\end{split}
\end{equation}
where $s_j$ solves \eqref{eq:motherODE_fire} for $j=1, \, 2, \, \ldots, \, n$. 
\end{lemma} 
\begin{proof} 
The $i$'th neuron fires during the time intervals $I_i^1, \, I_i^2, \, \ldots, \, I_i^{K(i)}$. That is, the right-hand side of the ODE in \eqref{eq:motherODE} is equal to one during these time intervals and zero otherwise. Hence, since we employ a Heaviside response function $\Phi$, 
\begin{equation} \label{eq:A1}
    \sum_j s_j(t-\tau_d) W_{ij} + B_i (t) \left\{ \begin{array}{cc}
       \geq 0,  & t \in \bigcup_{k=1}^{K(i)} I_i^k \\
       < 0,  & t \notin \bigcup_{k=1}^{K(i)} I_i^k, 
    \end{array} \right.
\end{equation}
which yields \eqref{eq:mother_inequalities}. 

The solution $s_i$ of \eqref{eq:motherODE_fire} is continuous as the derivative remains bounded. Due to \eqref{eq:A1}, we therefore conclude that \eqref{eq:mother_linear_system} must hold at any $t_i^k$ and $t_i^k+\Delta t_i^k$ where $B_i$ is also continuous as $t_i^k$ and $t_i^k+\Delta t_i^k$ constitute the endpoints of the interval $I_i^k$.
\end{proof}

Assume that the $i$'th neuron does not fire in the time interval $[0,T]$ under consideration. Then, even though the right-hand side of the ODE in \eqref{eq:motherODE_fire} is zero, the synaptic drive of this neuron could potentially have an impact on the solution of the network model \eqref{eq:motherODE}. On the other hand, it seems challenging to extract information about the associated connectivities from zero firing information.   
\begin{corollary}\label{co:no_firing} 
If neuron number $i$ does not fire on $(0,T]$, then the connectivities $W_{i1}, \, W_{i2}, \, \ldots, \, W_{in}$ satisfy 
\begin{equation*}
    \sum_j s_j(t-\tau_d) W_{ij}  < - B_i (t), \quad a.e. \ t \in (0,T]. 
\end{equation*}
\end{corollary}
\begin{proof}
    Follows from the second inequality in \eqref{eq:mother_inequalities} as neuron $i$ does not have any time intervals of firing.
\end{proof}

For any $\hat{i} \in \Gamma$, \eqref{eq:mother_linear_system} yields a linear system which can be solved separately from the others: the unknowns $W_{\hat{i} 1}, \, W_{\hat{i} 2}, \, \ldots , \, W_{\hat{i} n}$ only appear in the system arising when $i=\hat{i} \in \Gamma$. 

We may use \eqref{eq:mother_linear_system}, only considering the equations corresponding to $t=t_i^k$, to write up the linear system 
\begin{equation} \label{eq:baby_linear_system}
    \AA \ww = \bb, 
\end{equation}
where 
\begin{align*}
    \AA & = [a_{kj}^{(i)}] \in \mathbb{R}^{K(i) \times n}, \quad a_{kj}^{(i)} = s_j(t_i^k-\tau_d) \\
    \ww &= [W_{i1} \,  W_{i2} \, \ldots W_{in}]^T \in \mathbb{R}^{n \times 1}, \\
    \bb &= [b_k^{(i)}] \in \mathbb{R}^{K(i) \times 1}, \quad b_k^{(i)} = -B_i(t_i^k). 
\end{align*}
Using the notation  
\begin{align*}
    \tt_i &= [t_i^1-\tau_d, \, t_i^2-\tau_d, \, \ldots, \, t_i^{K(i)}-\tau_d]^T \in \mathbb{R}^{K(i) \times 1}, \\
    \ss_j(\tt_i) &= [s_j(t_i^1-\tau_d), \, s_j(t_i^2-\tau_d), \, \ldots, \, s_j(t_i^{K(i)}-\tau_d)]^T \in \mathbb{R}^{K(i) \times 1}, 
\end{align*} 
that is, $\tt_i$ contains the $\tau_d$-shifted start time instances of the firing events of the $i$'th neuron, and $\ss_j(\tt_i)$ contains the synaptic drive of neuron number $j$ at these time instances, we have
\begin{equation*}
    \AA = [ \ss_1(\tt_i) \, \, \ss_2(\tt_i) \, \, \ldots \, \, \ss_n(\tt_i) ]. 
\end{equation*}
Thus, $\ss_1(\tt_i), \, \ss_2(\tt_i), \, \ldots, \, \ss_n(\tt_i)$ constitute the columns of $\AA$. 

Invoking the rank theorem immediately yields the below corollary.

\begin{corollary} \label{co:nullspace}
    If $K(i) < n$, then $\AA$ has a nontrivial null space and \eqref{eq:baby_linear_system} cannot have a unique solution (or a unique least squares solution). Furthermore, when $K(i) = n$ the rows of $\AA$ are linearly independent if, and only if, the vectors $$\ss_1(\tt_i), \, \ss_2(\tt_i), \, \ldots, \, \ss_n(\tt_i),$$ are linearly independent.  
\end{corollary}
It thus follows that we cannot expect to uniquely recover $\ww$ from \eqref{eq:baby_linear_system} unless the $i$'th neuron fires $n$ times. In addition, the firing behaviour of the network has to be such that the columns of $\AA$ become linearly independent. 

\begin{remark} 
In principle, \eqref{eq:mother_linear_system} yields two linear equations for each choice of $k$, i.e., for each time interval of firing. The two vectors associated with these two equations read 
\begin{equation} \label{vec:almost_parallel}
\begin{split}
    &[s_1(t_i^k-\tau_d), \, s_2(t_i^k-\tau_d), \, \dots \,, \, s_n(t_i^k-\tau_d)], \\
    &[s_1(t_i^k+\Delta t_i^k-\tau_d), \, s_2(t_i^k+\Delta t_i^k-\tau_d), \, \dots \,, \, s_n(t_i^k+\Delta t_i^k-\tau_d)],
\end{split}   
\end{equation}
which would typically constitute two rows in a system matrix. Using standard theory, we conclude from \eqref{eq:motherODE_fire} that the functions $\{ s_i(t)\}$ are continuous and have uniformly bounded derivatives throughout $(0,T]$. Consequently, if $0 < \Delta t_i^k \ll 1$, the vectors \eqref{vec:almost_parallel} will typically be almost parallel, which could lead to a large condition number of $\AA$. Due to this possibility, we decided to not include the equations corresponding to $t=t_i^k + \Delta t_i^k$, $k=1, \, 2, \, \ldots, \, K(i)$, in system \eqref{eq:baby_linear_system}. 
\end{remark}

\subsection{Singular values} \label{sec:singular_values}
Inverse problems are typically ill-posed, and we have already seen that uniqueness might fail in Corollary \ref{co:nullspace}. The purpose of this subsection is to shed light on the nonzero singular values of $\AA$ since they play a fundamental role for revealing the stability properties of the minimum norm least squares solution
\begin{equation*}
    \ww = (\AA)^\dagger \bb
\end{equation*}
of \eqref{eq:baby_linear_system}, where $(\AA)^\dagger$ denotes the Moore-Penrose/pseudo-inverse of $\AA$. 
Hence, we want to analyze \eqref{eq:baby_linear_system} using the standard framework for linear inverse problems. More precisely, we will show that solving \eqref{eq:baby_linear_system} will become problematic if either the time between successive firing events becomes short or if synchronization occurs. We discuss the case of short time between firing events in detail. The case of synchronization can be derived in a similar manner, considering the columns instead of the rows of $\AA$.     

Recall that we have enumerated the firing events chronically \eqref{ineq:chronically}. Let 
\begin{align*}
    h_i &= \min_k (t_i^{k+1} - t_i^k), \\
    q_i &= \argmin_k (t_i^{k+1} - t_i^k), 
\end{align*}
i.e., $h_i$ represents the minimum time distance for neuron $i$ between two successive firing events and $q_i$ indicates the specific firing event for which this occurs. For simplicity, we omit the subscript $i$ of $h_i$ and $q_i$ in the following. We also use the notation $\ee_q$ for the $q$'th standard basis vector, and $\nullspace{[\AA]^T}$ represents the null space of $[\AA]^T$. 
The orthogonal projection of $\ee_{q+1}-\ee_q$ onto $\nullspace{[\AA]^T}^{\perp} = \textnormal{col}(\AA)$ is denoted $\widehat{\ee_{q+1}-\ee_q}$. 

\begin{theorem}\label{th:condA}
    Let $\sigma_1 \geq \sigma_2 \geq \ldots \geq \sigma_r > 0$ denote the nonzero singular values of $\AA$. Assume that $\ee_{q+1}-\ee_q \notin \nullspace{[\AA]^T}$. Then,   
\begin{align*}
    \sigma_1 &\geq c_1, \\
    \sigma_r &\leq \frac{c_2}{\| \widehat{\ee_{q+1}-\ee_q} \|} \, h,
\end{align*}
where 
\begin{align*}
    c_1 &= \sqrt{\sum_j \left( e^{-T} s_j^0 \right)^2}, \\
    c_2 &= \sqrt{\sum_j \left( s_j^0 + 1 \right)^2}. 
\end{align*}
Hence, if $\AA$ is invertible, its spectral condition number obeys 
\begin{equation*}
    \textnormal{cond} (\AA) \geq \frac{\sqrt{2} c_1}{c_2} \, h^{-1}. 
\end{equation*}
\end{theorem}
\begin{proof}
Consider the following two rows in the matrix $\AA$, 
\begin{equation*}
    \begin{split}
    \aa_q &= [s_1(t_i^q-\tau_d), \, s_2(t_i^q-\tau_d), \, \dots \,, \, s_n(t_i^q-\tau_d)], \\
    \aa_{q+1} &= [s_1(t_i^{q+1}-\tau_d), \, s_2(t_i^{q+1}-\tau_d), \, \dots \,, \, s_n(t_i^{q+1}-\tau_d)].
\end{split}   
\end{equation*}
The difference between these two vectors 
\begin{equation*}
    \dd = [d_1, \, d_2, \, \ldots, d_n] = \aa_{q+1} - \aa_q,
\end{equation*}
will have components which obey, due to the mean value theorem, 
\begin{equation*}
    |d_j| \leq D_j \, h \quad j=1, \, 2, \, \ldots, n,  
\end{equation*}
where
\begin{align*}
    \sup_{t\in[0,T]} \left|\frac{ds_j}{dt}(t) \right| \leq \max \left\{\sup_{t\in[0,T]}s_j(t), \sup_{t\in[0,T]} |1-s_j(t)| \right\}\leq 1+s_j^0 := D_j. 
\end{align*}
The last inequality follows from
\begin{align*}
    s_j^0e^{-t} \leq s_j(t) \leq s_j^0e^{-t}+1,
\end{align*}
which can be derived using \eqref{eq:motherODE} and remembering that $\Phi \in \{0,1\}$.
Using the Rayleigh quotient, and standard results, 
\begin{align*}
    \sigma_r^2 &= \min_{\xx \in \nullspace{[\AA]^T}^\perp} \frac{\left\langle \xx, \AA [\AA]^T \xx \right\rangle}{\| \xx \|^2} \\
    & \leq \frac{\left\langle [\AA]^T(\widehat{\ee_{q+1}-\ee_q}), [\AA]^T (\widehat{\ee_{q+1}-\ee_q}) \right\rangle}{\| \widehat{\ee_{q+1}-\ee_q} \|^2} \\
    & = \frac{\left\langle [\AA]^T(\ee_{q+1}-\ee_q), [\AA]^T (\ee_{q+1}-\ee_q) \right\rangle}{\| \widehat{\ee_{q+1}-\ee_q} \|^2} \\
    & = \frac{\| \aa_{q+1} - \aa_{q} \|^2 }{\| \widehat{\ee_{q+1}-\ee_q} \|^2} \\
    & \leq \frac{h^2 \sum_j D_j^2}{\| \widehat{\ee_{q+1}-\ee_q} \|^2} \\
    & = \frac{h^2 \sum_j (1+s_j^0)^2}{\| \widehat{\ee_{q+1}-\ee_q} \|^2}.
\end{align*}
Note that if $\AA$ is invertible, then $\widehat{\ee_{q+1}-\ee_q} = \ee_{q+1}-\ee_q$ and hence $\| \widehat{\ee_{q+1}-\ee_q} \|^2=2$.
Finally, since 
\begin{align*}
    s_j(t) \geq e^{-T} s_j^0, \quad t \in (-\tau_d,T], 
\end{align*}
\begin{align*}
    \sigma_1^2 &= \lambda_{\max} \left\langle \AA [\AA]^T \right\rangle \\
    &= \max_{\xx} \frac{\left\langle \xx, \AA [\AA]^T \xx \right\rangle}{\| \xx \|^2} \\
    & \geq \frac{\left\langle [\AA]^T \ee_q, [\AA]^T \ee_q \right\rangle}{\| \ee_q \|^2} \\
    & = \|\aa_{q} \|^2  \\
    & \geq \sum_j (e^{-T} s_j^0)^2. 
\end{align*}
\end{proof}

According to the above theorem, the firing intervals should (ideally) not appear too close in time as $h$ would become small and consequently, some rows of $\AA$ could become almost parallel. Then, the stability of the present parameter estimation problem deteriorates -- even if we decide to only search for least squares solutions of \eqref{eq:baby_linear_system} in the orthogonal complement of the null space of $\AA$.

\subsection{The algorithm}\label{subsec:algorithm}
For the sake of simplicity, we will not use the inequalities \eqref{eq:mother_inequalities} to potentially improve the recovery of the connectivities, or try to compute $ \left[ W_{ij} \right]_{j=1}^n$ for $i \notin \Gamma$. Recall that $\Gamma$ contains the indices corresponding to neurons that fire, cf. \eqref{set:neurons_that_fire}. The algorithm is presented using truncated SVD (TSVD) as the regularization method, see for example \cite{hansen2010discrete}. This method was chosen due to its simplicity. Note however that other methods for regularization could be used as well. 

For a given set of firing intervals $\{ I_i^k \}$, the algorithm for estimating the connectivities reads: 
\begin{enumerate}
    \item[a)] For $i=1,\, 2, \, \ldots, \, n$ solve the decoupled ODEs \eqref{eq:motherODE_fire}. This yields \\ $s_1(t), \, s_2(t), \, \dots, \, s_n(t)$.   
    \item[b)] For $i \in \Gamma$
    \begin{enumerate}[label = (\roman*)]
        \item Use $s_1(t), \, s_2(t), \, \dots, \, s_n(t)$ to compute the matrix entries of $\AA$.
        \item Use TSVD regularization to find an approximation to the solution of 
            \begin{equation} \label{regularized_problem}
                \min_{\ww} \| \AA \ww - \bb \|_2. 
            \end{equation}
    \end{enumerate}
\end{enumerate}
The regularized solution is then given by
\begin{equation}\label{eq:regularized_sol}
    \ww_{\kappa} = \sum_{j=1}^{\kappa} \frac{\langle \mathsf u^{(i)}_j,\bb\rangle}{\sigma^{(i)}_j}\mathsf v^{(i)}_j,
\end{equation}
where $\langle \mathsf u,\mathsf v\rangle = \mathsf u^T\mathsf v$, and $\sigma^{(i)}_j$, $\mathsf u_j^{(i)}$ and $\mathsf v_j^{(i)}$ are the singular values, the left singular vectors and right singular vectors of $\AA$, respectively. Note that the regularized solution \eqref{eq:regularized_sol} approaches the minimum norm least squares solution as $\kappa \rightarrow n$ when TSVD regularization is employed.

\subsection{Error considerations in inverse problems}\label{subsec:error}
As mentioned, inverse problems are usually ill-posed, meaning that small inaccuracies in the input data can lead to significant errors in the solution. This inherent instability is a fundamental characteristic of such problems, which directly impacts error behaviour and the interpretation of classical error bounds.

In our case, inaccuracies in the firing intervals introduce changes in the matrix $\AA$, and noise in the external input $B_i$ contributes changes to the vector $\bb$, in \eqref{eq:baby_linear_system}. Let $\ww$ and $\tilde{\mathsf{w}}^{(i)}$ be solutions to $\AA\ww = \bb$ and $\AA\tilde{\mathsf{w}}^{(i)} = \bb_\delta$, respectively. Here $\bb_\delta$ is a perturbed right-hand side. As long as $\AA$ is invertible and $\bb\neq \bf{0}$, a classical result, see for example \cite{lyche2020numerical}, states that 
\begin{equation} \label{eq:error_bound_b}
    \frac{1}{K(\AA)}\frac{\|\bb-\bb_\delta\|}{\|\bb\|} \leq \frac{\|\tilde{\mathsf{w}}^{(i)}-\ww\|}{\|\ww\|} \leq K(\AA)\frac{\|\bb-\bb_\delta\|}{\|\bb\|},
\end{equation}
where $K(\AA)$ is the condition number of $\AA$. Similarly, if $\AA$ is perturbed, i.e., $\tilde{\mathsf{w}}^{(i)}$ solves $\AA_\delta \tilde{\mathsf{w}}^{(i)} = \bb$, where $\AA_\delta = \AA + \mathsf{E}$ is the perturbed matrix, the relative error satisfies the bound
\begin{equation} \label{eq:error_bound_A}
    \frac{\|\tilde{\mathsf{w}}^{(i)}-\ww\|}{\|\ww\|} \leq \frac{K(\AA)}{1-\|(\AA)^{-1}\mathsf{E}\|}\frac{\|\mathsf{E}\|}{\|\AA\|},
\end{equation}
provided that $\|(\AA)^{-1}\mathsf{E}\| < 1$.

While \eqref{eq:error_bound_b} and \eqref{eq:error_bound_A} suggest useful bounds in theory, they rely heavily on $K(\AA)$. For inverse problems, the condition number $K(\AA)$ is typically large, such that these classical bounds provide no useful information. This behaviour is characteristic for inverse problems and reflects their inherent instability, as error amplification is expected and standard results about error propagation are often not useful for characterizing errors in the solution.

In our case, as demonstrated by Theorem \ref{th:condA}, the condition number of $\AA$ can become large, typically leading to significant error amplification. To mitigate this issue, it is standard to apply regularization methods, which reduce error amplification but introduce additional modelling errors. This creates a trade-off between approximating the solution of \eqref{eq:baby_linear_system} and avoiding critical error amplification.

Employing TSVD as the regularization method, the regularized solution is given by \eqref{eq:regularized_sol}, i.e., the regularized solution $\ww_\kappa$ is the solution of $\AA_\kappa \ww_\kappa = \bb$, where $\AA_\kappa = \sum_{j=1}^\kappa \mathsf{u}^{(i)}_j \sigma^{(i)}_j (\mathsf{v}^{(i)}_j)^T$. Let $\tilde{\mathsf{w}}^{(i)}_\kappa$ solve the perturbed system $\AA_\kappa \tilde{\mathsf{w}}^{(i)}_\kappa = \bb_\delta$. Using the triangle inequality, the total error between the true solution of \eqref{eq:baby_linear_system} and the regularized solution can be bounded as
\begin{equation}\label{eq:error_trig}
    \|\ww - \tilde{\mathsf{w}}^{(i)}_\kappa\| \leq \|\ww - \mathsf{w}_\kappa^{(i)}\| + \|\mathsf{w}_\kappa^{(i)} - \tilde{\mathsf{w}}^{(i)}_\kappa\|.
\end{equation}
This bound highlights the two sources of error: the regularization error and the error due to perturbations in the data. Increased regularization reduces the condition number of $\AA$ which in turn reduces the error amplification from perturbations in the data, see \eqref{eq:error_bound_b}. So, when $\kappa$ is decreased, i.e., increased regularization, the error from perturbations is reduced. However, the regularization error, the first term on the right-hand side in \eqref{eq:error_trig}, is increased. On the other hand, increasing $\kappa$ has the opposite effect. Thus, choosing an appropriate $\kappa$ is crucial for achieving the best possible compromise between the two sources for error.

The method for choosing the regularization parameter $\kappa$ depends on the source of the noise. The idea behind the selection method is to avoid including the SVD components where the noise dominates, to prevent overfitting the noise. For noise in $\bb$, Morozov's discrepancy principle, see for example \cite[Sec.~5.2]{hansen2010discrete}, could be used to select $\kappa$. Specifically, $\kappa$ is chosen such that
\begin{equation}\label{eq:Dsicrep_b}
    \| \AA \ww_\kappa - \bb_\delta\|_2 \geq \|\bb - \bb_\delta\|_2 > \| \AA \ww_{\kappa+1} - \bb_\delta\|_2.
\end{equation}

One crucial assumption for the accuracy of Morozov's discrepancy principle is that the noise only occurs in $\bb$. Therefore, when noise occurs in $\AA$, an adjusted discrepancy principle is required. First note that
\begin{equation*}
    \|\AA_\delta \ww - \bb\|_2 = \|(\AA_\delta - \AA) \ww\|_2 \approx \|(\AA_\delta - \AA)\ww_\kappa\|_2,
\end{equation*}
where we have used that $\bb=\AA\ww$. An adjusted discrepancy principle is then as follows: choose $\kappa$ such that
\begin{equation}\label{eq:Dsicrep_fire}
    \| \AA_\delta \ww_\kappa - \bb\|_2 \geq \|(\AA_\delta - \AA)\ww_\kappa\|_2 > \| \AA_\delta \ww_{\kappa+1} - \bb\|_2.
\end{equation}

A more thorough error investigation for the problem considered in this manuscript would require detailed information about the singular value distribution and precise information about the noise\footnote{Considering synthetic data, this is not a problem as the noise is given, but in practice this is a bottleneck.}. Information about the singular value distribution is at the moment out of reach due to the nonlinearity of the forward map \eqref{def:forwardOperator} and the form of the entries of $\AA$, cf. \eqref{eq:baby_linear_system}.

\section{Continuity of the forward operator} \label{subsec:continuity_forward}
In this section we establish the aforementioned continuity property of the forward operator $F$ defined in \eqref{def:forwardOperator}. We first prove local continuity under a threshold condition, before investigating sufficient criteria on the external input guaranteeing that the condition holds.

Whether neuron $i$ fires or not is determined by the 
sign of $G_i(t-\tau_d)+B_i(t)$, where $$G_i(t) = \sum_j W_{ij}s_j(t),$$ is the contribution from all the neurons in the network and $s_j$ denotes the solutions of \eqref{eq:motherODE}. It is therefore convenient to define the following operator for $t\in[0,T]$,
\begin{align}\label{def:Goperator}
Q_t: (\R^{n\times n}, \|\cdot\|_{\max})  \quad & \rightarrow \quad  (\Gspace, \|\cdot\|_\Gspace ), \\
\WW \quad & \mapsto \quad \mathsf G := [G_1, G_2, \dots, G_n], \nonumber
\end{align}
where 
$$\Gspace = (L^\infty([-\tau_d,t]))^n \quad \mathrm{and} \quad \| \mathsf G \|_{\Gspace} = \max_i \|G_i\|_{L^\infty([-\tau_d,t])}.$$
Before proving the (local) continuity of $F$ we must show the continuity of $Q_T$.  First, by multiplying \eqref{eq:motherODE} by $W_{ij}$ and summing over $j$, we observe that $G_i$ solves
\begin{subequations} \label{eq:G-ODE}
\begin{align}\label{eq:G-ODE-diff}
\frac{dG_i}{dt}(t) + G_i(t) &= \sum_j W_{ij}\Phi \left( G_j(t-\tau_d) + B_j(t) \right), \quad t \in (0,T],\\
G_i(t)&=e^{-t}\sum_j W_{ij} s_j^0, \quad t \in [-\tau_d,0]. \label{eq:G-ODE-init}
\end{align}
\end{subequations}
The dynamics in the ODE \eqref{eq:G-ODE-diff} change when a jump occurs in the sum on the right-hand side of the equation. It will become evident in the upcoming proof of the continuity of $Q_T$ that it is useful to separate the times where a jump could occur in the right-hand side of \eqref{eq:G-ODE-diff} from the times where no jump occurs. Due to the Heaviside function $\Phi$, a jump can only occur around times $t$ where $|G_i(t-\tau_d)+B_i(t)|$ is small. We therefore define the following two sets,
\begin{align}\label{eq:def-beta}
\begin{split}
    \beta_{\gamma}^i(z) &= \{ \ t \in [0,z] \ | \ |G_i(t-\tau_d) +B_i(t)| \leq \gamma^{-1}\}, \quad \gamma\in\mathbb{N},\\
    \beta^i(z) &= \{ \ t \in [0,z] \ | \ G_i(t-\tau_d) +B_i(t) = 0\}.
\end{split}
\end{align} 
A crucial observation about $\beta_\gamma^i$ which we will use in the upcoming proof is stated in the next remark.
\begin{figure}
    \centering
    \includegraphics[width=0.95\linewidth]{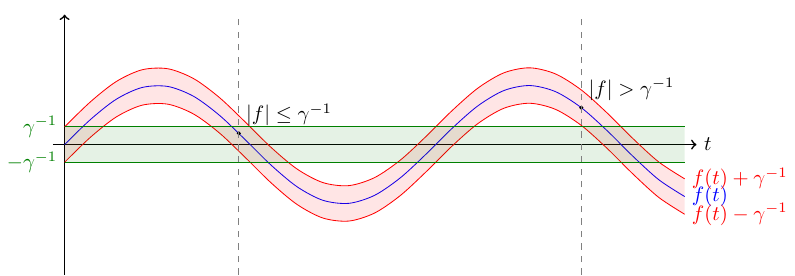}
    \caption{Illustration of Remark \ref{remark:beta}. When $f$ (in blue) satisfies $|f(\tilde t)| >\gamma^{-1}$ for a time $\tilde t$, i.e., the point $(\tilde t, f(\tilde t))$ is outside the green area, any function $g(t)$ satisfying $|f(t)-g(t)|<\gamma^{-1}$, i.e., the graph of $g(t)$ is within the red area, will have the same sign as $f(t)$ at $t=\tilde t$.}
    \label{fig:beta}
\end{figure}
\begin{remark}\label{remark:beta}
   If the blue curve in Figure \ref{fig:beta} is the graph of $G_i(t-\tau_d) + B_i(t)$, then $\beta_\gamma^i(z)$ contains all times $t\in[0,z]$ where $\left(t,G_i(t-\tau_d) +B_i(t)\right)$ is located within the green area.
   Furthermore, let $\tilde G_i$ be a different function satisfying 
    \begin{equation}\label{eq:remark_G}
        \left|G_i(t-\tau_d)-\tilde G_i(t-\tau_d)\right|\leq \frac{1}{\gamma}
    \end{equation} 
    for $t\in[0,z]$. Then, 
    \begin{align*}
        \left|G_i(t-\tau_d) + B_i(t) - \left(\tilde G_i(t-\tau_d) + B_i(t)\right) \right|
        &\leq \frac{1}{\gamma}.
    \end{align*}
    Consequently, the graph of $\tilde G_i(t-\tau_d)+B_i(t)$ will be within the red area and we can conclude the following:
    \begin{itemize}
        \item if $t \notin \beta^i_{\gamma}\big(z\big)$, $G_i(t-\tau_d)+ B_i(t)$ and $\tilde G_i(t-\tau_d) + B_i(t)$ are of the same sign and not equal to 0,
        \item if $t \in \beta^i_{\gamma}\big(z\big)$, $G_i(t-\tau_d)+ B_i(t)$ and $\tilde G_i(t -\tau_d) + B_i(t)$ might be of opposite sign.
    \end{itemize}
\end{remark} 

We are now ready to show that $Q_T$ defined in \eqref{def:Goperator} is continuous. As mentioned, we will need a threshold condition. More specifically, we assume that $\lim_{\gamma \to \infty} |\beta_{\gamma}^i(T)|=0$ for 
$i=1,2,\ldots,n$. 
\begin{lemma}\label{lem:Q_cont}
Let $T>0$ and assume that $ \lim_{\gamma \to \infty} |\beta_{\gamma}^i(T)| =0$ for $i=1,2,\ldots,n$ for the choice $\WW\in\R^{n\times n}$. Then $Q_T$ defined in \eqref{def:Goperator} is continuous at $\WW$. 
\end{lemma}
\begin{proof}
To show continuity of $Q_T$, we will use an induction argument to show that $Q_{m\tau_d}$ is continuous for $m=0,1,\ldots T/\tau_d$. If $T/\tau_d\notin\mathbb N$ the proof still holds true with a minor adjustment. First, let $\tilde \WW$ be a perturbation of $\WW$ and define $H_i(t)=G_i(t)-\tilde G_i(t)$, where $G_i(t)$ and $\tilde G_i(t)$ denote the solutions of \eqref{eq:G-ODE} with $\WW$ and $\tilde \WW$, respectively. The difference $H_i(t)$ solves
\begin{align}\label{eq:H_ode}
\begin{split}
\frac{dH_i}{dt}(t) + H_i(t) &= \sum_j W_{ij}\Phi \big( G_j(t-\tau_d) + B_j(t) \big
)\\
& \qquad  - \sum_j\tilde W_{ij}\Phi \big(\tilde G_j(t-\tau_d) + B_j(t) \big), \quad t \in (0,T],\\
H_i(t)&= e^{-t}\sum_j(W_{ij}-\tilde W_{ij}) s_j^0, \quad t \in [-\tau_d,0].
\end{split}
\end{align}
For the base case in the induction argument we set $m=0$, i.e., $t\in[-\tau_d,0]$ in \eqref{eq:H_ode}, and thus
\begin{equation*}
    |H_i(t)|\leq \|s_j^0\|_{\ell^1}\|\WW-\tilde \WW\|_{\max}.
\end{equation*}
Taking the maximum with respect to $i$, we conclude that $Q_0$ is continuous at $\WW$. 

Now, let $t\in[(m-1)\tau_d,m\tau_d]$ for $m\geq1$ in \eqref{eq:H_ode}, and pick an arbitrary $\varepsilon>0$. Assume that $Q_{(m-1)\tau_d}$ in \eqref{def:Goperator} is continuous at $\WW$. Thus, for all $\gamma\in \mathbb{N}$ there exists a $\delta_1>0$ such that $|G_j(t-\tau_d)-\tilde G_j(t-\tau_d)|<\gamma^{-1}$ for all $j$ and $t\in[0,m\tau_d]$ when $0<\|\WW- \tilde \WW \|_{\max} <\delta_1$. So, $\tilde G_j$ satisfies condition \eqref{eq:remark_G} in Remark \ref{remark:beta}, and the difference
\begin{equation*}
    \Phi \big(G_j(t-\tau_d)+B_j(t) \big) - \Phi \big(\tilde G_j(t-\tau_d)+B_j(t) \big), \ t \in [0,m\tau_d],
\end{equation*}
can be nonzero only at times when $|G_j(t-\tau_d) + B_j(t)|<\gamma^{-1}$, i.e., at times included in the set $\beta_{\gamma}^j(m\tau_d)$. Due to the assumption $\lim_{\gamma\rightarrow\infty}|\beta^i_\gamma(T)|= 0$, there exists $ \eta \in \mathbb N$ s.t. 
\begin{equation} \sum_j  |\beta_{\gamma}^j(m\tau_d)| <\frac{\varepsilon}{2\|\WW\|_{\max}}\quad\forall\ \gamma\in \mathbb{N} \ \textnormal{ s.t. } \ 0<\frac{1}{\gamma}\leq\frac{1}{\eta}. \label{eq:eta}
\end{equation} 

Let $\delta_2 = \min\left\{\delta_1,\frac{\varepsilon}{2(\|s_j^0\|_{\ell^1}+n)}\right\}$, where $\delta_1$ is such that $\gamma$ satisfies \eqref{eq:eta}, and solve \eqref{eq:H_ode} for $t\in[0,m\tau_d]$. Then, we get the following bound
\begin{align*}
    |&H_i(t)| \\
&\leq e^{-t}|H_i(0)| \\
&\quad+ e^{-t} \left|\int_{0}^te^z\sum_j W_{ij}\left(\Phi\big(G_j(z-\tau_d)+B_j(z)\big)-\Phi\big(\tilde G_j(z-\tau_d)+B_j(z)\big)\right) dz\right|\\
&\quad+ e^{-t}\left|\int_{0}^te^z\sum_j\big(W_{ij}-\tilde W_{ij}\big) \Phi\big(\tilde G_j(z-\tau_d)+B_j(z)\big) dz\right| \\
&\leq e^{-t}|H_i(0)| + e^{-t} \sum_j |W_{ij}|\int_{\beta^j_{\gamma}}e^z dz + e^{-t}\sum_j\big|W_{ij}-\tilde W_{ij}\big| \int_0^te^z dz\\
&\leq (\|s_j^0\|_{\ell^1}+n)\big\|\WW-\tilde \WW\big\|_{\max} + \|\WW\|_{\max}\sum_j  |\beta_{\gamma}^j(m\tau_d)| <\varepsilon,
\end{align*}
whenever $\|\WW-\tilde\WW\|_{\max}<\delta_2$. Taking the maximum with respect to $i$ completes the proof of continuity.
\end{proof}

Equipped with the above lemma we now show that the forward operator $F$ in \eqref{def:forwardOperator} is (locally) continuous under the assumption $\lim_{\gamma\rightarrow\infty}|\beta_\gamma^i(T)|=0$ for $i=1,2,\ldots,n$.
\begin{theorem}
Under the assumption in Lemma \ref{lem:Q_cont}, the forward operator $F$ in \eqref{def:forwardOperator} is continuous at $\WW$. 
\end{theorem}
\begin{proof}
    Let $\{I_i^k\}$ and $\{\tilde{I}_i^k\}$ be the firing intervals due to the connectivities $\WW$ and $\tilde{\WW}$, respectively, i.e., $\{I_i^k\} = F(\WW)$ and $\{\tilde{I}_i^k\}=F(\tilde{\WW})$. Here, $\tilde{\WW}$ is a perturbation of $\WW$. Furthermore, denote $A_i= \bigcup_k I_i^k$ and $\tilde A_i = \bigcup_k \tilde I_i^k$. Observe that if $t\in A_i\triangle \tilde{A}_i$, then $t$ is either in $A_i$ or in $\tilde{A}_i$, but not in both. In addition, from \eqref{eq:mother_inequalities} we have that $t\in \bigcup_k I_i^k$ if and only if $G_i(t-\tau_d)+B_i(t) \geq 0$, and similarly for $t\in \bigcup_k \tilde I_i^k$. Thus, using the metric defined in \eqref{def:symdiff_metric}, based on \eqref{def:general_sym_diff}, we can conclude that
    \begin{align*}
        d(&\{I_i^k\},\{\tilde{I}_i^k\}) \\
        &= \sup_i \left|(A_i \setminus \tilde{A}_i) \cup  (\tilde{A}_i \setminus A_i)\right| \\
        &=\sup_i\left|\left\{t\in[0,T] \mid \sgn\big(G_i(t-\tau_d)+B_i(t)\big) \neq \sgn\big(\tilde{G}_i(t-\tau_d)+B_i(t)\big)\right\}\right|.
    \end{align*}
    
    Let $\varepsilon>0$. From the assumption $\lim_{\gamma\rightarrow\infty}|\beta_\gamma^i|=0$, we have that there exists $\eta\in\mathbb N$ such that $|\beta_\gamma^i(T)|<\eps$ for all $i\in\{1,2,\ldots,n\}$ for $\gamma>\eta$.
    Furthermore, from Lemma \ref{lem:Q_cont} we know that $G_i$ is continuous at $\WW$. So, for $\gamma>\eta$ there exists $\delta>0$ such that 
    \begin{equation*}
        \|\WW-\tilde \WW\|_{\max}<\delta\implies\|G_i-\tilde G_i\|_\infty <\frac{1}{\gamma}, \quad \forall i\in\{1,2,\ldots,n\}.
    \end{equation*}
    Thus, $\tilde G_i$ satisfies \eqref{eq:remark_G} in Remark \ref{remark:beta} and we have that $G_i(t-\tau_d)+B_i(t)$ and $\tilde G_i(t-\tau_d)+B_i(t)$ can only have opposite signs if $t\in\beta_\gamma^i(T)$. This implies that 
    \begin{equation*}
        \left\{t\in[0,T] \mid \sgn\big(G_i(t-\tau_d)+B_i(t)\big) \neq \sgn\big(\tilde{G}_i(t-\tau_d)+B_i(t)\big)\right\}\subseteq \beta_\gamma^i(T),
    \end{equation*}
    and consequently $\left|(A_i\setminus\tilde A_i)\cup (\tilde A_i\setminus A_i)\right|\leq |\beta_\gamma^i(T)|<\varepsilon$ when $\gamma>\eta$. 
    So, taking the supremum with respect to $i$, we conclude that $F$ is continuous at $\WW$.
\end{proof}
\begin{remark}
    If we scale the connectivities $\left[ W_{ij} \right]$ with $n^{-1}$, which is necessary for the limit $n\rightarrow\infty$ in \eqref{eq:motherODE} to be well-defined, the proof of continuity of the forward operator will be uniform in the number of neurons $n$. 
\end{remark}
The local continuity of the forward operator $F$ is proven under the assumption that $\beta^i_\gamma(T)\rightarrow0$ as $\gamma\rightarrow\infty$ for $i=1,2,\ldots,n$, see \eqref{eq:def-beta}. 
Below we investigate criteria on the external inputs $\{ B_i \}$ for this to be the case. The next lemma states sufficient conditions for $\{ B_i \}$. The main idea of the proof rests on two observations. First, if the argument of $\Phi$ in \eqref{eq:motherODE} crosses 0 at most a finite number of times, $\beta^i(T)$ in \eqref{eq:def-beta} has to be of zero measure. Second, looking at \eqref{eq:G-ODE}, if there is an interval where no neurons change from firing to not firing or vice versa, $G_i(t)$ will be of the form $c_1e^{-t}+c_2$. This is reflected in the two assumptions on $\{B_i\}$ stated in the lemma.
\begin{lemma}\label{lem:assumptionsB}  Assume that for all $i\in\{1,2,\ldots,n\}$, $B_i(t)$ is a function which satisfies:
\begin{itemize}[leftmargin=6.7em]
    \item[Assumption 1.] The function $\Phi\left(c_1 e^{-(t-\tau_d)}+c_2 + B_i(t)\right)$ has a finite number of jump discontinuities on $[0,T]$.
    \item[Assumption 2.] $\left|\{t\in [0,T] \mid c_1e^{-(t-\tau_d)}+c_2+B_i(t) = 0\}\right|=0$. 
\end{itemize}
Then $ \lim_{\gamma \to \infty} |\beta_{\gamma}^i(T)| = |\beta^i(T)|=0$ for all $i$.
\end{lemma}

\begin{proof}
The proof is divided into three steps. In Step 1 we establish by induction that $G_i(t)$ is indeed piecewise of the form $c_1 e^{-t} + c_2$. The last two steps consist of showing that $|\beta^i(T)|=0$ and $|\beta^i(T)| = \lim_{\gamma\rightarrow\infty} |\beta_\gamma^i(T)|$. 

\noindent{\bf Step 1:}
We show that $G_i(t)$ is piecewise of the form $c_1 e^{-t}+c_2$ for $t\in[-\tau_d,m\tau_d]$, $m=0,1,2,\ldots,T/\tau_d$ by an induction argument. As before, the argument still holds true if $T/\tau_d\notin \mathbb N$ by a minor adjustment.

The exact statement of the induction hypothesis is: for all $i=1,\ldots,n$, $G_i(t)$ is piecewise of the form $c_1 e^{-t}+c_2$ for $t\in[-\tau_d,m\tau_d]$ with a finite number of points of non-differentiability. Note that whenever any of the Heavisides in \eqref{eq:G-ODE} switch between 0 and 1, we have a point of non-differentiability.

Assume that the hypothesis is true up to $m$. Then, there exists constants $a_i$ and $b_i$ such that
\begin{equation*}
    G_i(t-\tau_d) = a_i e^{-(t-\tau_d)} + b_i, \quad t\in[m\tau_d,(m+1)\tau_d],
\end{equation*} 
between two points where $G_i(t)$ is non-differentiable. From Assumption 1 we then have that 
$\Phi(G_i(t-\tau_d)+B_i(t))$ has a finite number of jump discontinuities on $[m\tau_d,(m+1)\tau_d]$. Thus, by Definition \ref{def:firing} every neuron fires a finite number of times in that interval. 
So, we can rewrite \eqref{eq:G-ODE} as 
\begin{equation}\label{eq:G-ODE-finite}
    \frac{d G_i}{dt} + G_i = \sum_{j=1}^n W_{ij} \sum_{k=1}^{K(j)} \mathcal{X}_{I_j^k}(t),\quad t \in [m\tau_d,(m+1)\tau_d],
\end{equation}
where $I_j^k$ are the firing intervals for neuron $j$.
The solution $G_i(t)$ of \eqref{eq:G-ODE-finite} on $t\in[m\tau_d,(m+1)\tau_d]$ is
\begin{align}\label{eq:solG}
\begin{split}
    G_i(t) &= G_i(m\tau_d)e^{m\tau_d-t} \\
    &\qquad + \sum_j W_{ij}\sum_k\left[\Phi(t-t_j^k-\Delta t_j^k)(e^{t_j^k+\Delta t_j^k}-e^{t_j^k})-\mathcal{X}_{I_j^k}(t)e^{t_j^k}\right]e^{-t} \\
    &\quad+ \sum_j W_{ij}\sum_k\mathcal{X}_{I_j^k}(t) \\
    &:= C_1^{i}(t)e^{-t} + C_2^i(t).
\end{split}
\end{align}
Let $\mathcal{T} =\{t\in [m\tau_d,(m+1)\tau_d] \mid t=t_j^k \text{ or } t=t_j^k+\Delta t_j^k\}$ denote the set of all times in $[m\tau_d,(m+1)\tau_d]$ that mark either the start or the end of a firing interval in the network.
Then for all $t\in \mathcal{T}^c$, $C_1^i(t)$ and $C_2^i(t)$ are constant.
Consequently, $G_i(t)$ continues to be piecewise of the form $c_1 e^{-t}+c_2$ with a finite number of points of non-differentiability on $[m\tau_d,(m+1)\tau_d]$, and thus so is the case for $G_i(t)$ on the whole interval $[-\tau_d,(m+1)\tau_d]$. 

In the base case $m=0$, $t\in[-\tau_d,0]$, and from \eqref{eq:G-ODE-init} we have that $G_i(t) =\sum_j W_{ij} s_j^0e^{-t}$. Clearly, we can write $G_i(t)=c_1 e^{-t}+c_2$, with $c_1 = \sum_j W_{ij} s_j^0$ and $c_2=0$. 
\\
{\bf Step 2:}
Let $d_i^l$, $l=1,\ldots,L(i)$ be the points of non-differentiability for $G_i(t-\tau_d)$. From Step 1 we have that for each $(d_i^l,d_i^{l+1})$, $l=1,\ldots L(i)-1$, there exists $c_1$ and $c_2$ such that $G_i(t-\tau_d) = c_1e^{-t}+c_2$. In addition, we observe that
\begin{equation*}
    \{t\in(d_i^l,d_i^{l+1}) \mid c_1 e^{-t}+c_2+B_i(t) = 0\}\subset \{t\in[0,T] \mid c_1 e^{-t}+c_2+B_i(t) = 0\}.
\end{equation*}
Thus, from Assumption 2 we can conclude that 
\begin{equation*}
    |\{t\in(d_i^l,d_i^{l+1}) \mid c_1 e^{-t}+c_2+B_i(t) = 0\}| = 0,
\end{equation*}
and consequently $\beta^i(T)$, defined in \eqref{eq:def-beta}, is a finite union of sets with measure zero, and therefore, $|\beta^i(T)|=0$.\\
{\bf Step 3:}
We observe that $\beta_{\gamma}^i(T) \supseteq \beta_{\gamma+1}^i(T)$ for each $\gamma$ and $|\beta_{1}^i(T)|\leq T$. Then by continuity from above \cite[Proposition 14]{RoydenH} we have 
\begin{equation*}
    \left|\beta^i(T)\right| = \left|\bigcap_{\gamma=1}^\infty \beta_\gamma^i(T)\right| = \lim_{\gamma\rightarrow\infty}\left|\beta_\gamma^i(T)\right|.
\end{equation*}

Combining Step 2 and 3, we finally get that 
\begin{equation*}
    \lim_{\gamma\rightarrow\infty} |\beta_\gamma^i(T)| = |\beta^i(T)| = 0.
\end{equation*}
\end{proof}
\begin{remark}
    By solving \eqref{eq:G-ODE} and using that $0\leq \Phi(G_j(t-\tau_d) + B_j(t))\leq 1$, we get
    \begin{align*}
         W^- + \left(\sum_{j} W_{ij} s_j^0 - W^- \right)e^{-t}  \leq G_i(t) \leq W^+ + \left(\sum_{j} W_{ij} s_j^0 - W^+ \right)e^{-t}.
    \end{align*}
    It is an open problem whether this can be used to refine the restriction on the functions $\{B_i\}$ in Lemma \ref{lem:assumptionsB}.
\end{remark}

Which class of functions that satisfy the assumptions in Lemma \ref{lem:assumptionsB} remains unclear. However, in the subsequent corollary, we demonstrate that if $B_i$ is an extended piecewise analytic function satisfying $B_i\notin \textnormal{span} \{1,e^{-t}\}$ on any interval, the assumptions hold. We define an extended piecewise analytic function as follows.  

\begin{definition}[Extended piecewise analytic function]\label{def:extended_pw_analytic}
    A function $f$ is called \emph{extended piecewise analytic} on $[0,T]$, if for each open interval $Q_l = (a_l,b_l)$, in a set of a finite number of disjoint open intervals $Q_1, Q_2,\ldots,Q_m$ satisfying $[0,T] = \cup_l \overline{Q_l}$, there exists an associated function $f_l$ which is analytic on $(a_l-\eps_l, b_l+\eps_l)$ for some $\eps_l>0$ and satisfies $f|_{Q_l} = f_l|_{Q_l}$.
\end{definition}

\begin{corollary}\label{cor:assumptionsB} For all $i\in\{1,2,\ldots,n\}$, assume that $B_i$ is an extended piecewise analytic function on $[0,T]$ and that $B_i\notin \textnormal{span} \{1,e^{-t}\}$ on any interval.
Then $\lim_{\gamma \to \infty} |\beta_{\gamma}^i(T)| = |\beta^i(T)|=0$ for all $i$.
\end{corollary}
\begin{proof}
    By Definition \ref{def:extended_pw_analytic} we have that there exists $Q_l=(a_l,b_l)$ and $f_l$, $l=1,2,\ldots,m$ such that $B_i|_{Q_l}=f_l|_{Q_l}$.
    Suppose that $c_1e^{-t}+c_2+B_i(t)$ has infinitely many zeros $\{z_j\}$ in $Q_l$. Consequently, $c_1e^{-t}+c_2+f_l(t)$ has infinitely many zeros $\{z_j\}$ in $(a_l-\varepsilon_l,b_l+\varepsilon_l)$. The Bolzano-Weierstrass theorem then implies that $\{z_j\}$ has a cluster point in $(a_l-\varepsilon_l,b_l+\varepsilon_l)$. Furthermore, using the identity theorem for analytic functions we can conclude that $c_1e^{-t}+c_2 +f_l(t) \equiv 0$ on $(a_l-\varepsilon_l,b_l+\varepsilon_l)$. This implies that $c_1e^{-t}+c_2+B_i(t)\equiv 0$ on $(a_l,b_l)$, i.e., $B_i|_{Q_l}\in \textnormal{span}\{e^{-t},1\}$, which is a contradiction. Thus, $c_1 e^{-t}+ c_2 + B_i(t)$ has a finite number of zeros, and consequently $\Phi(c_1 e^{-t}+ c_2 + B_i(t))$ has a finite number of jump discontinuities on $(a_l,b_l)$.
    
    Thus, we can conclude that $\{t\in [0,T] \mid c_1e^{-t}+c_2+B_i(t) = 0\}$ consists of a finite number of isolated points and consequently $$|\{t\in [0,T] \mid c_1e^{-t}+c_2+B_i(t) = 0\}| = 0.$$ 
    So, the assumptions in Lemma \ref{lem:assumptionsB} are satisfied, and we can conclude that $\lim_{\gamma \to \infty} |\beta_{\gamma}^i(T)| = |\beta^i(T)|=0$ for all $i$.
\end{proof}
In Corollary \ref{cor:assumptionsB} we assumed that $B_i$ cannot be constant on any interval. However, as $B_i$ could be piecewise constant in real life experiments, we consider this special case on its own.
\begin{lemma}
    Assume that the range of $B_i$ equals $\{0,c_i\}$, where $c_i$ is a constant, that $B_i$ has a finite number of jump discontinuities, and that
    \begin{equation}\label{eq:c_assumption}
        c_i \neq -\sum_j W_{ij} v_j \textnormal{ for all binary vectors } \mathsf{v}=(v_1,v_2,\ldots,v_n).
    \end{equation}
    Then $\lim_{\gamma \to \infty} |\beta_{\gamma}^i(T)| = |\beta^i(T)|=0$ for all $i$.
\end{lemma}

\begin{proof}
    As in Step 1 in the proof of Lemma \ref{lem:assumptionsB}, we can conclude that $G_i(t)$ is piecewise of the form $c_1 e^{-t} + c_2$. Therefore, if $B_i(t)$ is a constant, $|\beta^i(T)|=|\{t\in[0,T]\mid G_i(t-\tau_d)+B_i(t)=0\}|$ can only be nonzero if $G_i(t-\tau_d)$ is constant and equal to $-B_i(t)$ on a subinterval. Thus, assume that $G_i(t-\tau_d)$ is constant on the interval $(a,b)$. This means that for $t\in(a-\tau_d,b-\tau_d)$, $dG_i/dt = 0$ and from \eqref{eq:G-ODE-diff} we get
    \begin{equation*}
        G_i(t) = \sum_j W_{ij} \Phi(G_j(t-\tau_d)+B_j(t)), \quad t\in(a-\tau_d,b-\tau_d).
    \end{equation*}
    Since $\Phi$ is the Heaviside function, the right-hand side of the above equation can be written as $\sum_j W_{ij} v_j$, where $\mathsf{v} = (v_1,v_2,\ldots,v_n)$ is a binary vector. So, from \eqref{eq:c_assumption} we can conclude that $G_i(t-\tau_d)$ will never be equal to $-B_i(t)$ and consequently $|\beta^i(T)|=0$.
\end{proof}

\section{Numerical experiments}\label{sec:num}
In this section we illustrate the potential to numerically reconstruct the connectivity matrix $\WW$ by solving the inverse problem defined in \eqref{eq:baby_linear_system} for each neuron $i$ using the algorithm presented in Section \ref{subsec:algorithm}. The firing intervals will be generated synthetically by solving the system of ODEs \eqref{eq:motherODE} using the true connectivity and given initial data $\{s_i^0\}$. The aim is to investigate the impact of noise on the inverse solution given these firing intervals.

The upcoming results are simulated using $n=20$ and $n=100$ neurons with simulation time $T=500$ and $T=2000$, respectively. The other parameters are as follows: the time step is set to $dt=1/500$, the time delay to $\tau_d=1$, and the external input to $B_i(t)=0.1$ for $i=1,\ldots,n$. The initial conditions $\{s_i^0\}$ are uniformly drawn from $(0,1)$.
 
From the derivation of the inverse problem in Section \ref{sec:method}, we observe that there are at least two possible sources for noise. 
Firstly, we note that the external input $B_i(t)$ can consist of both input from neurons outside of the network and from input from external sources. Thus, it is interesting to investigate how noise added to $\bb$ affects the inverse solution. Secondly, when constructing system \eqref{eq:baby_linear_system} we use the firing intervals. When dealing with real data, the start and end points of the firing intervals may be imprecise. This will be simulated by adding noise to the start and end points of the firing intervals. Generally, such noise will affect both $\AA$ and $\bb$. However, since $B_i(t)$ is set to be constant for $i=1,\ldots,n$, the noise added to the firing intervals will only affect $\AA$ in the upcoming examples. 

The noise is drawn from a normal distribution with zero mean and standard deviation $\psi$. The size of $\psi$ depends on whether noise is added to $\bb$ or to the firing intervals $\{I_i^k\}$. In addition, define $nl$ to be the noise level. For noise added to $\bb$ we use $\psi = \max \bb \cdot nl$ and for noise added to the firing intervals we use $\psi = \overline{\Delta t}\cdot nl$, where $\overline{\Delta t}$ is the median length of the firing intervals. We consider noise levels $nl=1\%$, $nl=5\%$ and $nl=10\%$.

Noise is added as follows in the two different cases. Denote the noise-free matrix and noise-free right-hand side in \eqref{eq:baby_linear_system} by $\AA$ and $\bb$, respectively. Let $\ww$ denote the true connectivity, i.e., $\AA\ww =\bb$. For the noisy right-hand side we use $\bb_\delta = \bb +\psi \boldsymbol\eta^{(i)}$, where the noise-vector $\boldsymbol\eta^{(i)}$ has components $\eta_k^{(i)}\sim\mathcal{N}(0,1)$, for $k=1,...,K(i)$, which are drawn independently from the standard normal distribution. 
For the firing intervals, we define $\hat I_j^k$ to be a perturbation of $I_j^k$ where noise is added to the start and end points: $\hat I_j^k = [\widehat{t_j^k},\widehat{t_j^k+\Delta t_j^k}]$, where $\widehat{t_j^k} = t_j^k +\psi \eta_j^k$ and similarly for $\widehat{t_j^k+\Delta t_j^k}$. Here, $\eta_j^k\sim\mathcal{N}(0,1)$ is drawn independently from the standard normal distribution. Note that the firing intervals where $\Delta t_j^k < \psi$ will be discarded. Finally, $\hat s_j$ denotes the solution of \eqref{eq:motherODE_fire} with the noisy firing intervals, $\{\widehat{I_j^k}\}$, and we define the noisy matrix as $\AA_\delta = [(a^{(i)}_\delta)_{kj}] = [\hat{s}_j(\widehat{t_i^k}-\tau_d)]$. 

As mentioned in Section \ref{subsec:algorithm}, we use TSVD as a regularization method. To determine the regularization parameter $\kappa$ we use Morozov's discrepancy principle \eqref{eq:Dsicrep_b} for the case where noise is added to $\bb$ and the adjusted discrepancy principle \eqref{eq:Dsicrep_fire} presented in Section \ref{subsec:error} when noise is added to the firing intervals.

\begin{figure}
    \centering
    \subfigure[Symmetric connectivity]{\includegraphics[width=0.45\linewidth]{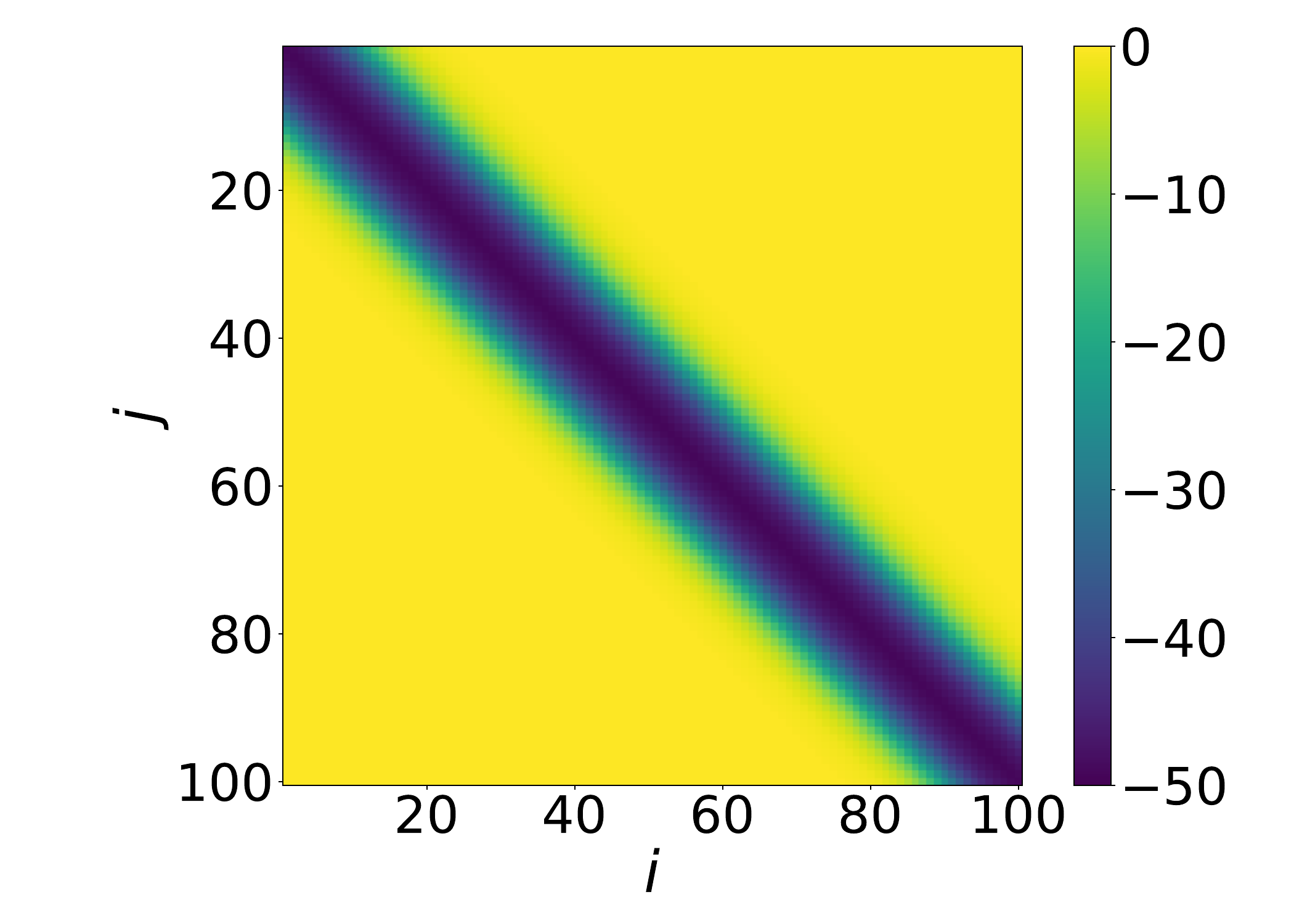}\label{fig:W_sym}}\qquad
    \subfigure[Non-symmetric connectivity]{\includegraphics[width=0.45\linewidth]{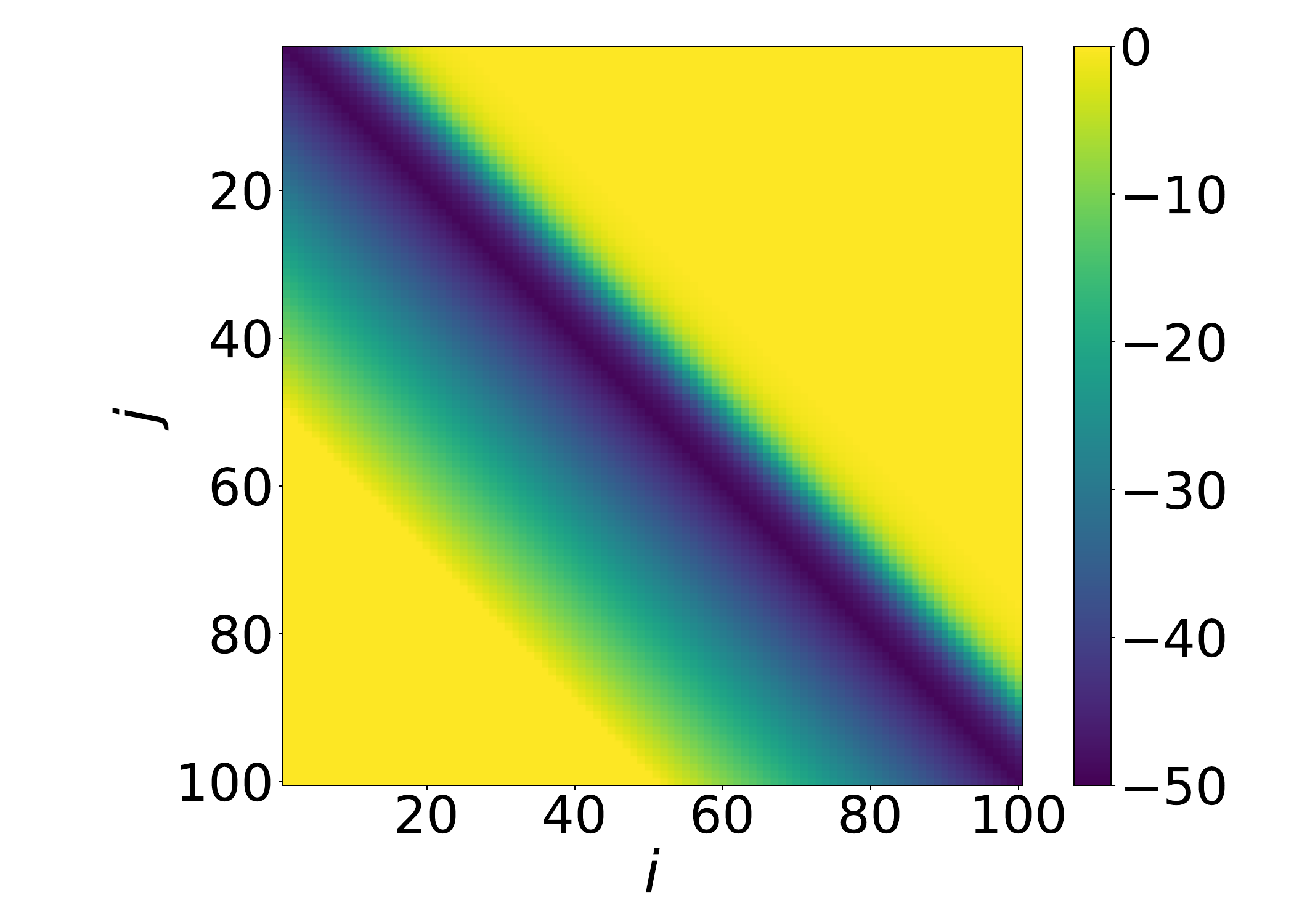}\label{fig:W_nonsym}}\\
    \caption{Samples of the true connectivity function for $n=100$.}\label{fig:W}
\end{figure}

In order to generate synthetic observation data as firing intervals, which will be used to construct the matrices $\{\AA\}$, the ODE \eqref{eq:motherODE} with the true $\WW$ is solved numerically with Euler's method. In doing so we introduce a time discretization error such that the start and end points of the firing intervals obtained are not necessarily the exact values where $\sum_j W_{ij} s_j(t-\tau_d) + B_i(t)$ intersect the $t$-axis. More precisely, the start and end points can deviate from the exact value by (at most) $dt$. This in turn introduces a discretization error when constructing $\AA$. Consequently, we avoid doing an \emph{inverse crime}. 

The condition number $K(\AA)$ of $\AA$ will typically be large because we consider an inverse problem, see Section \ref{subsec:error}. This is confirmed by Table \ref{tab:condA}, which presents the mean, maximum, and minimum values of $K(\AA)$ across various scenarios. Note that these values are quite large, particularly in the symmetric case with $n=100$ neurons. Consequently, the results we obtain for the inverse solution may not be highly accurate: the ill-posed nature of the problem yields severe error amplification that must be reduced by regularization, leading to additional modelling errors.

We will consider two different examples: one with a symmetric connectivity function in Section \ref{sec:sym} given by 
\begin{equation}\label{eq:Wsym}
    W(x,y) = -25\left(1+\tanh(2-20|x-y|)\right),
\end{equation}
and one with a non-symmetric connectivity function in Section \ref{sec:nonsym} given by
\begin{equation}\label{eq:Wnonsym}
    W(x,y) =\begin{cases}
         -25\left(1+\tanh(2-20|x-y|)\right), \quad x<y, \\
         25\left(1+\tanh{2}\right)\left(\frac{100}{49}(x-y) -1\right)\quad y \leq x < y+\frac{49}{100}, \\
         0, \quad \text{elsewhere}.
    \end{cases}
\end{equation}

In order to compare the true connectivity to the inverse solution we sample the connectivity function at $n^2$ points uniformly distributed on the square $[-0.5,0.5]^2$. For $n=100$ the samples of the symmetric connectivity function and the samples of the non-symmetric connectivity function are shown in Figure \ref{fig:W_sym} and \ref{fig:W_nonsym}, respectively. Along the $x$-axis the points are indexed by $i=1,\ldots,n$ and along the $y$-axis by $j=1,\ldots,n$. In addition, for consistency, all the numerical inverse solutions are plotted in the scale of the true connectivity.

 \begin{table}
    \centering
    \begin{tabular}{|c|c|c|c|c|c|c|}
    \hline
       & \multicolumn{3}{c|}{Symmetric}  &  \multicolumn{3}{c|}{Non-symmetric}\\
      \hline
         & Mean & Max & Min & Mean & Max & Min\\
      \hline
        $n=20$ & $2.9\cdot 10^3$ & $7.2\cdot10^3$ & $9.4\cdot10^2$ &  $1.4\cdot10^4$ & $3.5\cdot10^4$ & $3.4\cdot10^3$\\
         \hline
        $n=100$ &$4.6\cdot 10^{14}$ & $4.6\cdot10^{16}$ & $6.9\cdot10^2$&  $5.3\cdot 10^4$&  $1.1\cdot10^5$ & $2.3\cdot10^4$\\
        \hline
    \end{tabular}
    \caption{The mean, maximum and minimum condition number of $\{\AA\}$.}
    \label{tab:condA}
\end{table}

\subsection{Mildly or severely ill-posed}\label{sec:mildvssevere}
We make a small detour before presenting the examples in order to numerically investigate the degree of ill-posedness of the problem at hand. In Section \ref{sec:singular_values} we showed that the singular values $\{ \sigma_m \}$ of the matrix $\AA$ in \eqref{eq:baby_linear_system} may become small. Here we present plots of $\{ \sigma_m \}$. Recall that it is the speed of the decay of the singular values which determines whether a problem is mildly or severely ill-posed \cite{engl1996regularization}. 
\begin{figure}
    \centering
    \subfigure
    {\includegraphics[width=0.45\linewidth]{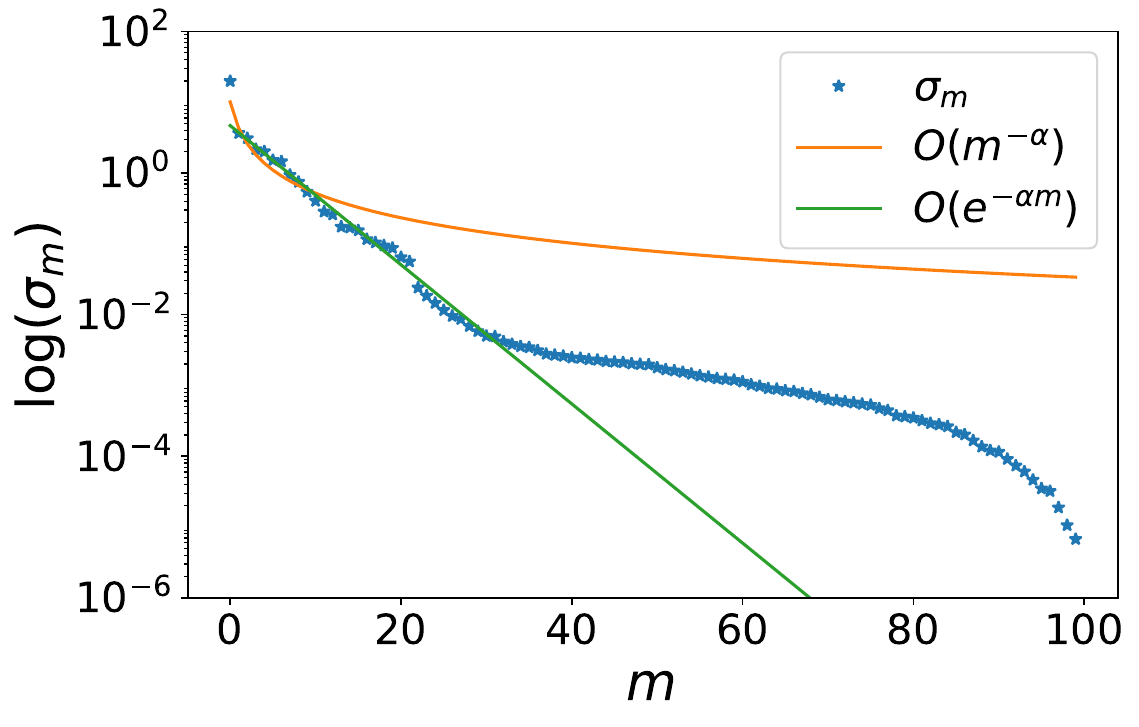}}\qquad
    \subfigure
    {\includegraphics[width=0.45\linewidth]{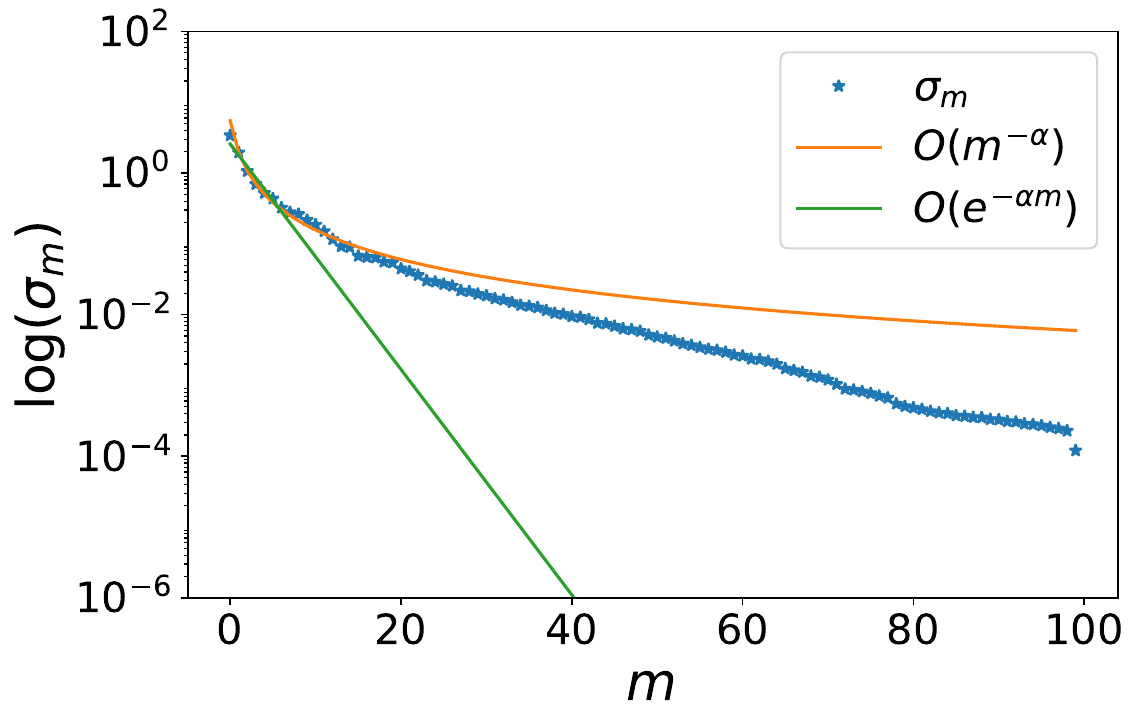}}\\
    \caption{Log-plots of the singular values of $\AA$ and their approximations $C m ^{-\alpha}$ and $C e^{-\alpha m}$ for neuron $i=3$. The left and right plots show these quantities for the symmetric and non-symmetric connectivities, respectively.}
    \label{fig:m_vs_s}
\end{figure}

In the left plot in Figure \ref{fig:m_vs_s} the singular values of $\AA$ for neuron $i=3$, in addition to the curves of $C m^{-\alpha}$ and $Ce^{-\alpha m}$ that approximate the decreasing singular values, are plotted for the symmetric connectivity \eqref{eq:Wsym}. The corresponding quantities for the non-symmetric connectivity \eqref{eq:Wnonsym} are depicted in the right plot. We note that similar plots were obtained for the other neurons as well.

It is difficult to draw any firm conclusions from Figure \ref{fig:m_vs_s}, however, we can observe that the singular values for the symmetric case decay faster than for the non-symmetric case. For the symmetric connectivity, problem \eqref{eq:baby_linear_system} seems to be severely ill-posed. On the other hand, for the non-symmetric connectivity, \eqref{eq:baby_linear_system} appears to be closer to mildly ill-posed. This suggests that we will obtain better inverse solutions with the non-symmetric connectivity than with the symmetric connectivity. 

\begin{figure}
    \centering
    \subfigure[Noise level: $1\%$. Regularization parameter: $\kappa=16.$]
    {\includegraphics[width=0.27\linewidth]{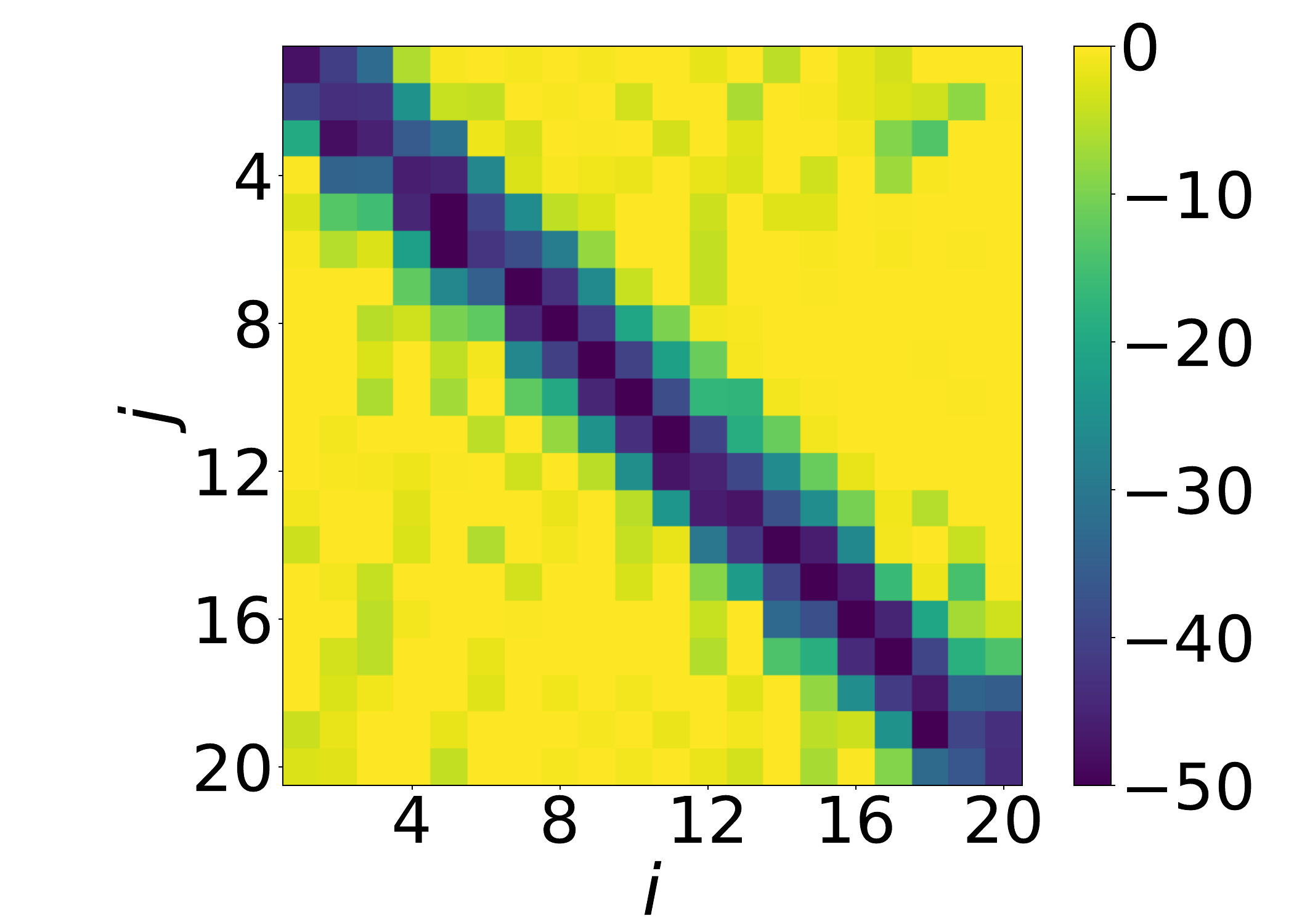}}\qquad
    \subfigure[Noise level: $5\%$. Regularization parameter: $\kappa=13$.]
    {\includegraphics[width=0.27\linewidth]{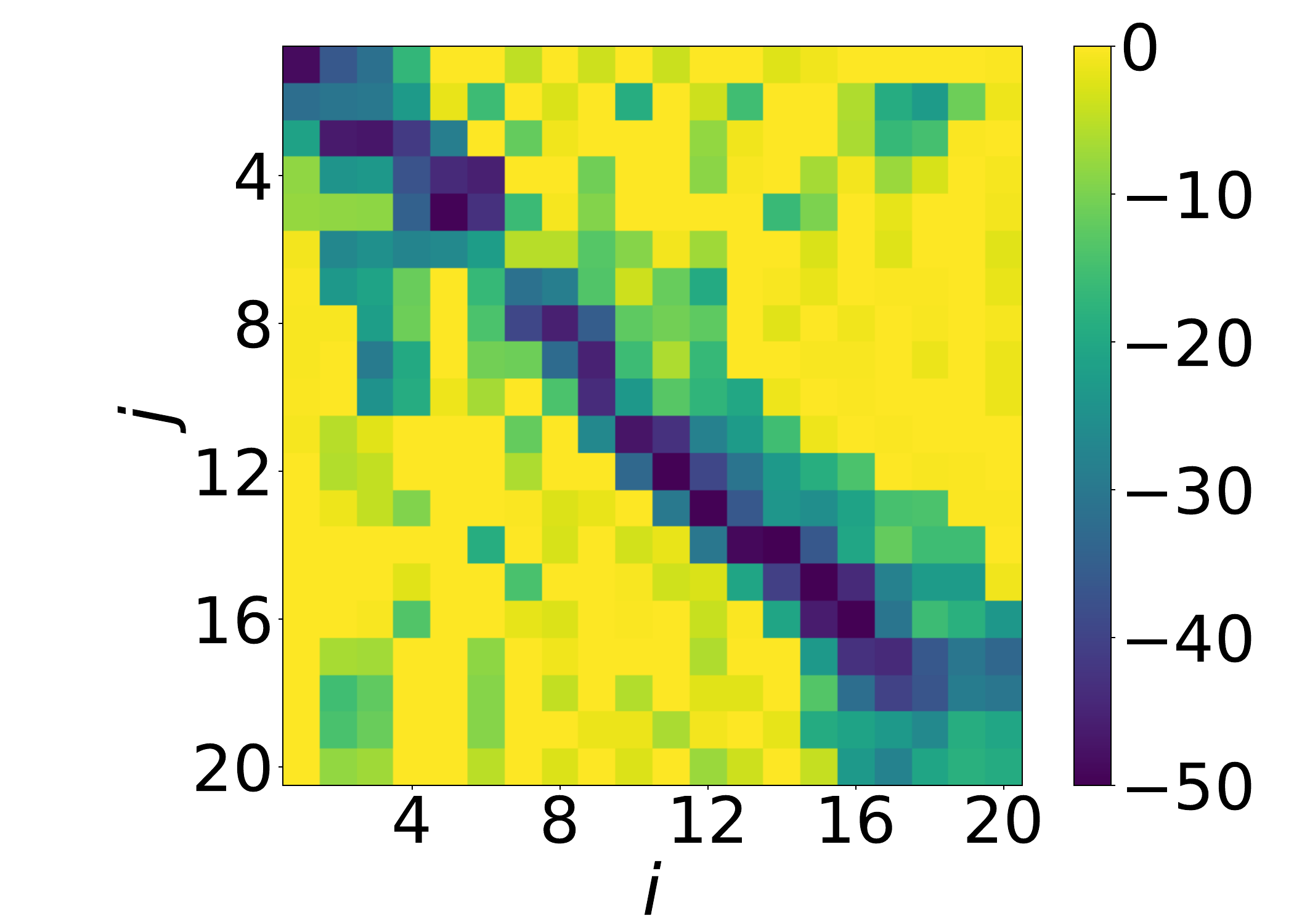}}\qquad
    \subfigure[Noise level: $10\%$. Regularization parameter: $\kappa=12$.]
    {\includegraphics[width=0.27\linewidth]{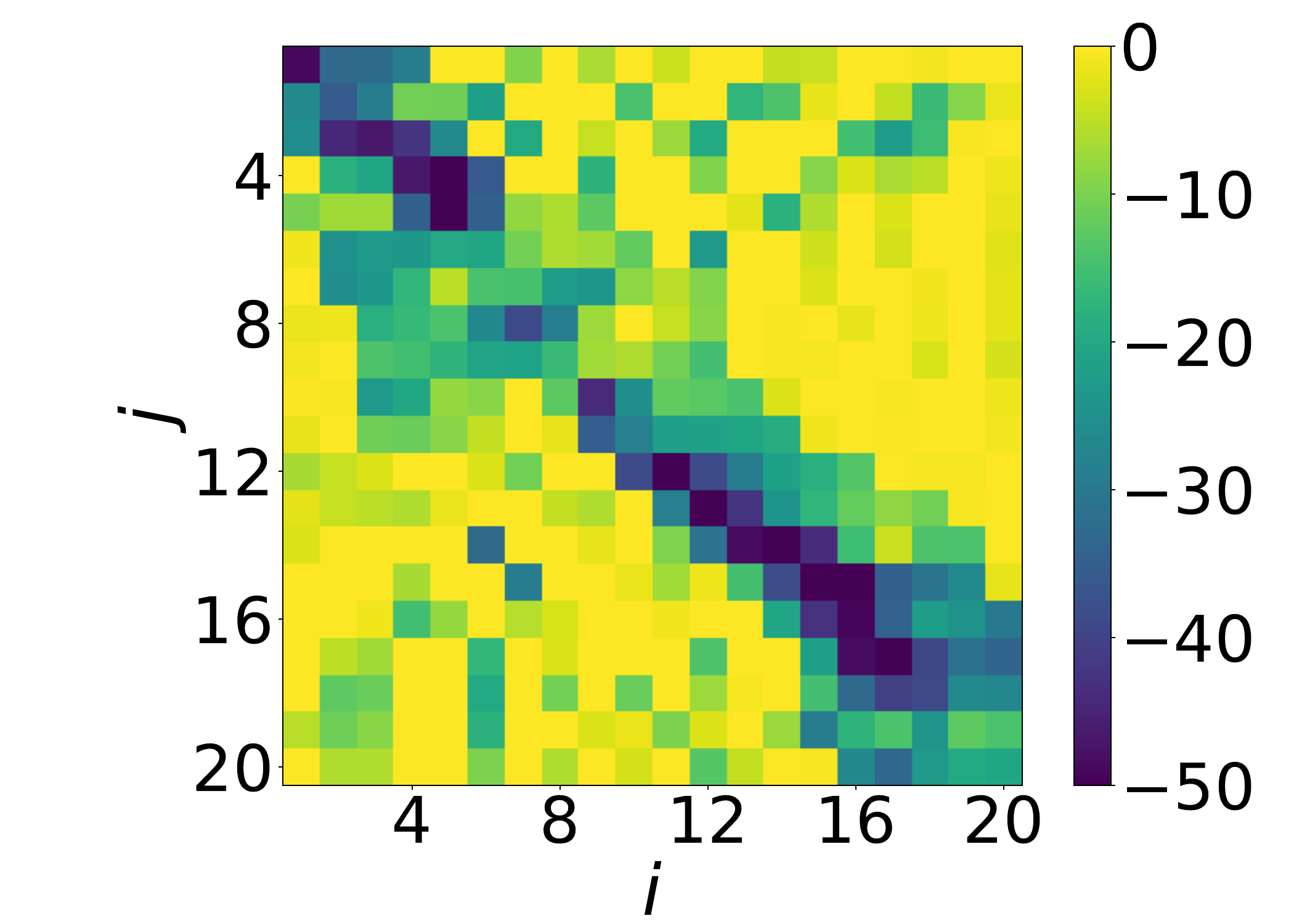}}\\
    \subfigure[Noise level: $1\%$.]{\includegraphics[width=0.27\linewidth]{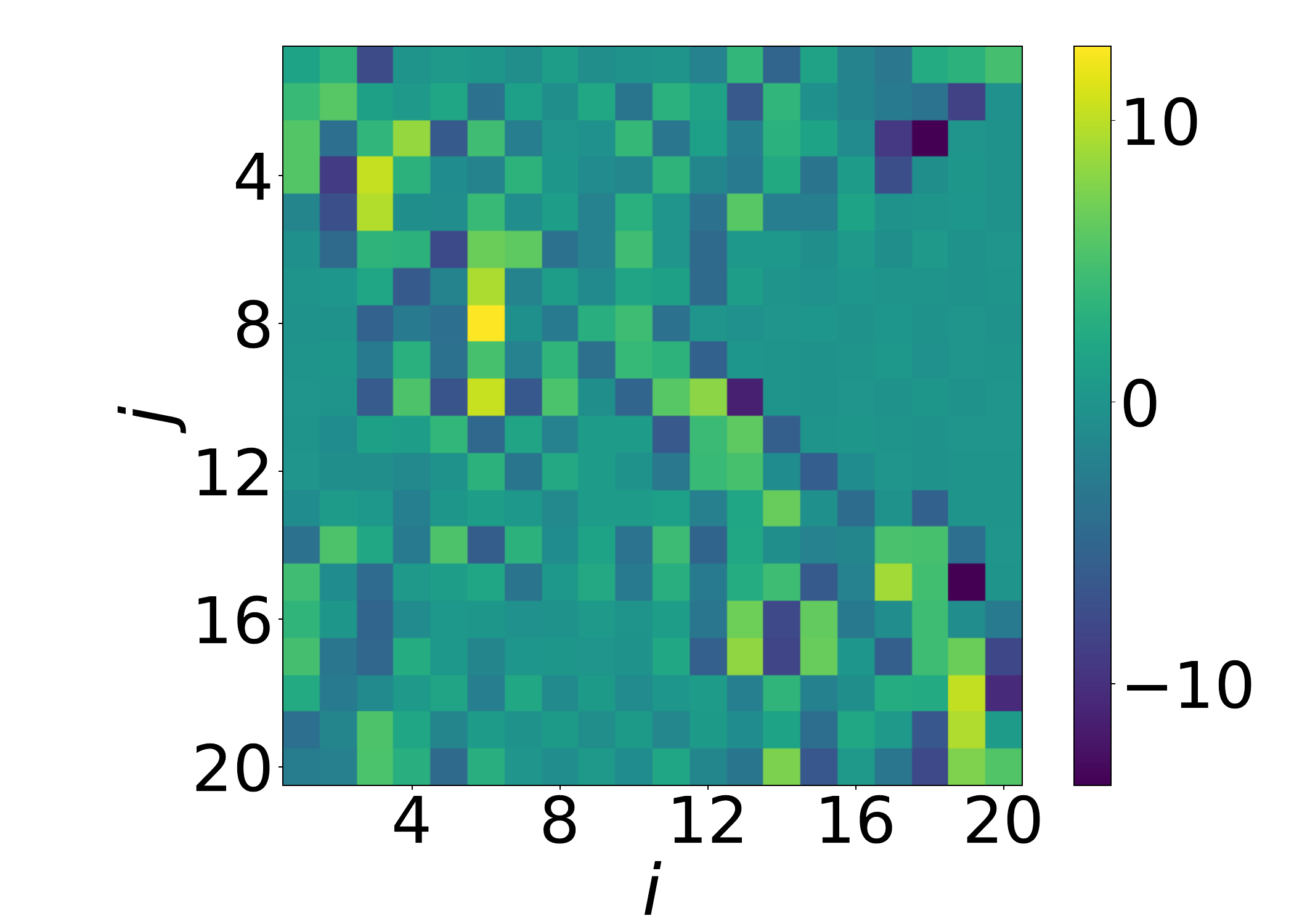}}\qquad
    \subfigure[Noise level: $5\%$.]{\includegraphics[width=0.27\linewidth]{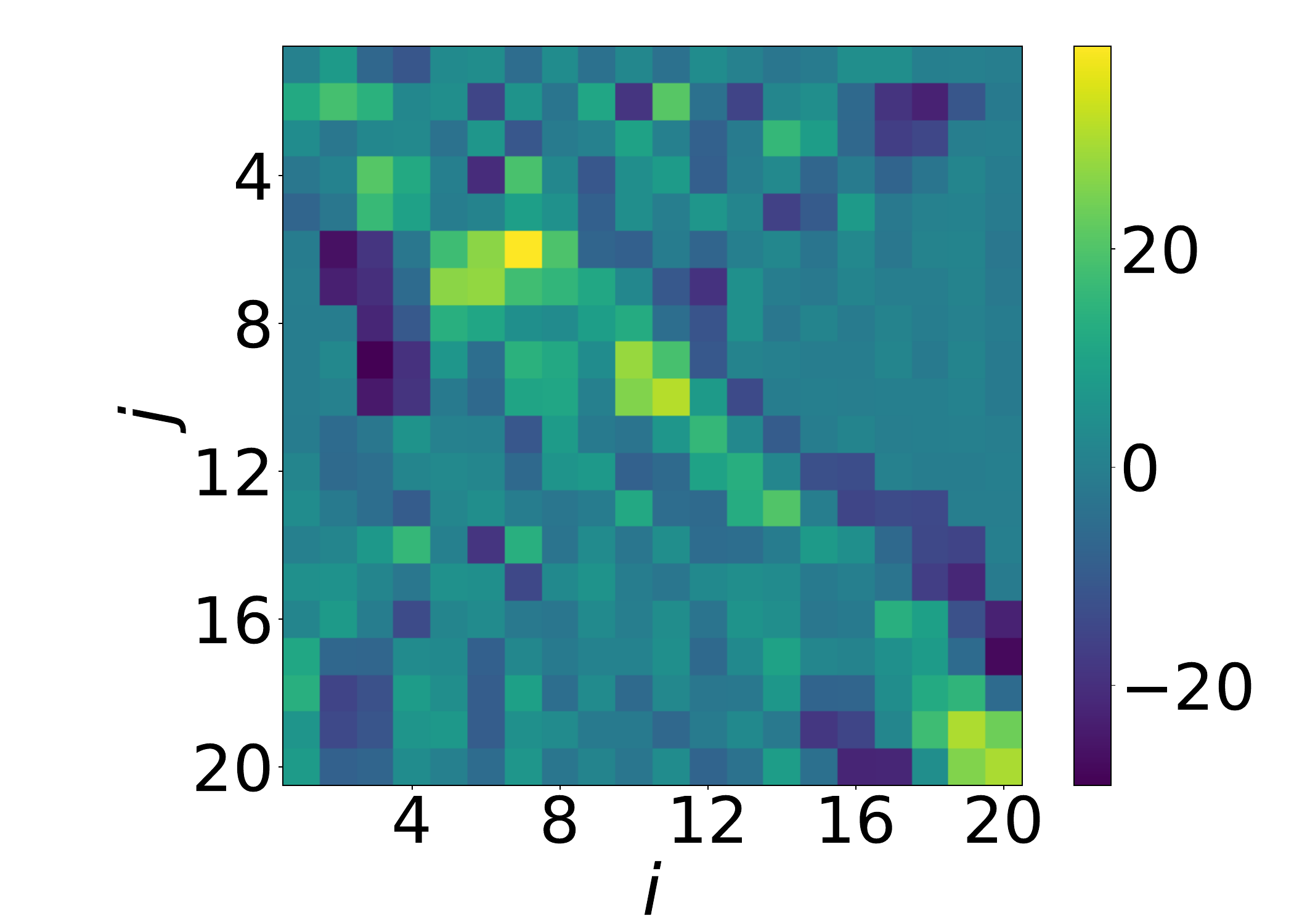}}\qquad
    \subfigure[Noise level: $10\%$.]{\includegraphics[width=0.27\linewidth]{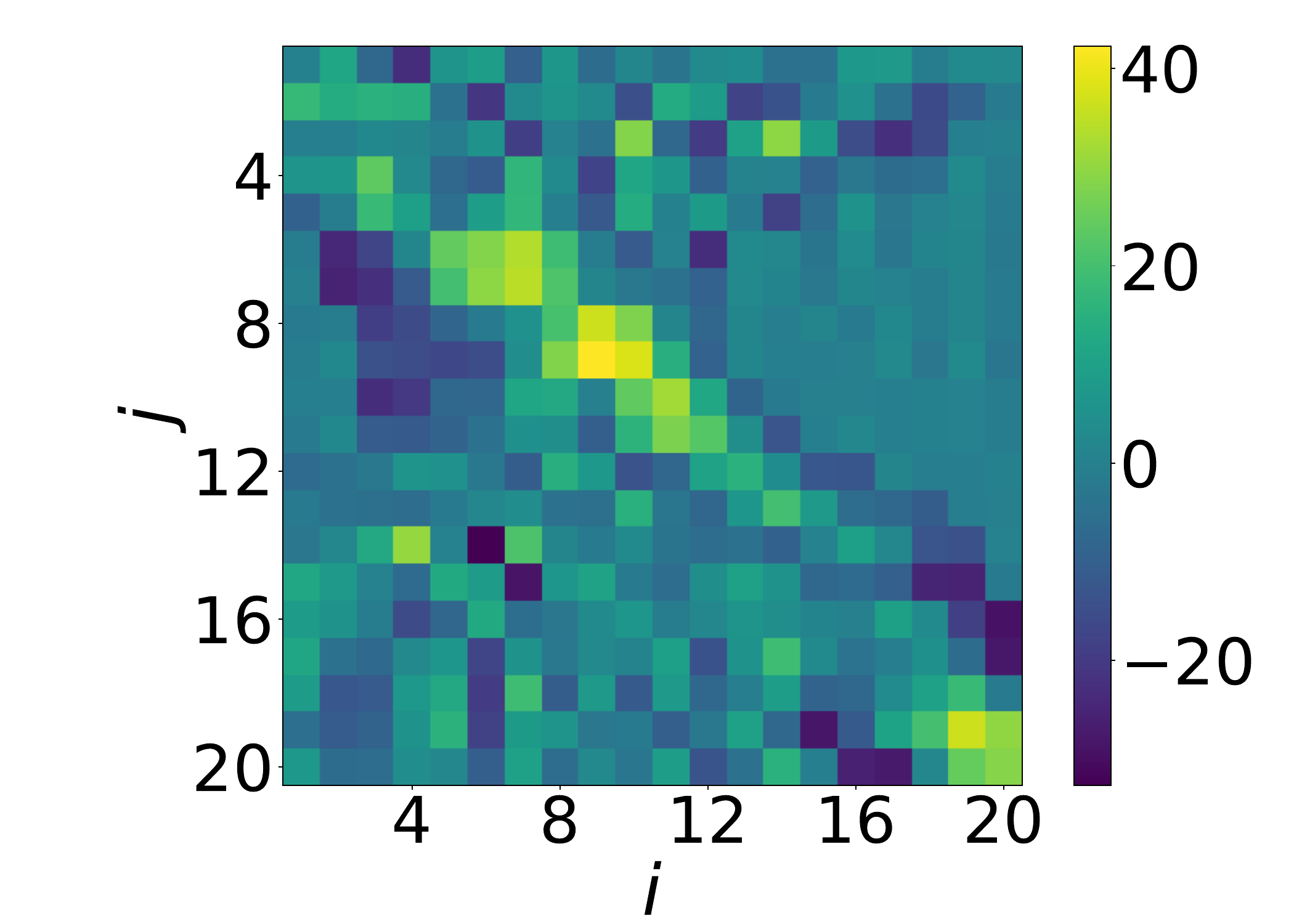}}\\
    \caption{Results for $n=20$ neurons with noise added to $\bb$ and where the true connectivity is symmetric. The upper row shows the inverse solution $\WW_{\text{inv}}^\kappa$ computed with various degree of noise. The lower row displays $\WW_{\text{inv}}^\kappa-\WW$.}
    \label{fig:nx20_noiseb_symmetric}
\end{figure}

\subsection{Symmetric connectivities}\label{sec:sym}
We now consider the inverse problem where the true connectivity is generated by the symmetric function \eqref{eq:Wsym}, see Figure \ref{fig:W_sym}. 

We first add noise to $\bb$ and explore how this affects the inverse solution $\WW^{\kappa}_{\text{inv}}$\footnote{The columns in $\WW^\kappa_{\text{inv}}$ are given by $\ww_\kappa$, see \eqref{eq:regularized_sol}.}. In Figure \ref{fig:nx20_noiseb_symmetric} the inverse solution and the difference $\WW^{\kappa}_{\text{inv}}-\WW$ is plotted for $n=20$ neurons. We observe that the inverse solution gets less accurate when we increase the noise, and, from the lower row of plots, that the magnitude of the error increases. In addition, we note that the regularization parameter $\kappa$ used to compute the inverse solutions becomes smaller when we increase the noise level. 

When increasing the number of neurons to $n=100$, we found that Morozov's discrepancy principle was unsuccessful in obtaining a reasonable choice for the regularization parameter for the $nl=1\%$ noise level. For a real world problem, it would not be possible to minimize the error in Frobenius norm, i.e., $\argmin_{\kappa} \|\WW^{\kappa}_{\text{inv}}-\WW\|_F,$ to obtain the regularization parameter $\kappa$, since we would not know the true solution $\WW$. In spite of this, we include the inverse solution calculated using this method to obtain the regularization parameter for the $1\%$ noise level, see Figure \ref{fig:nx100_noiseb_symmetric}. We present this result to demonstrate the existence of a satisfactory solution. 

\begin{figure} 
    \centering
    \subfigure{\includegraphics[width=0.27\linewidth]{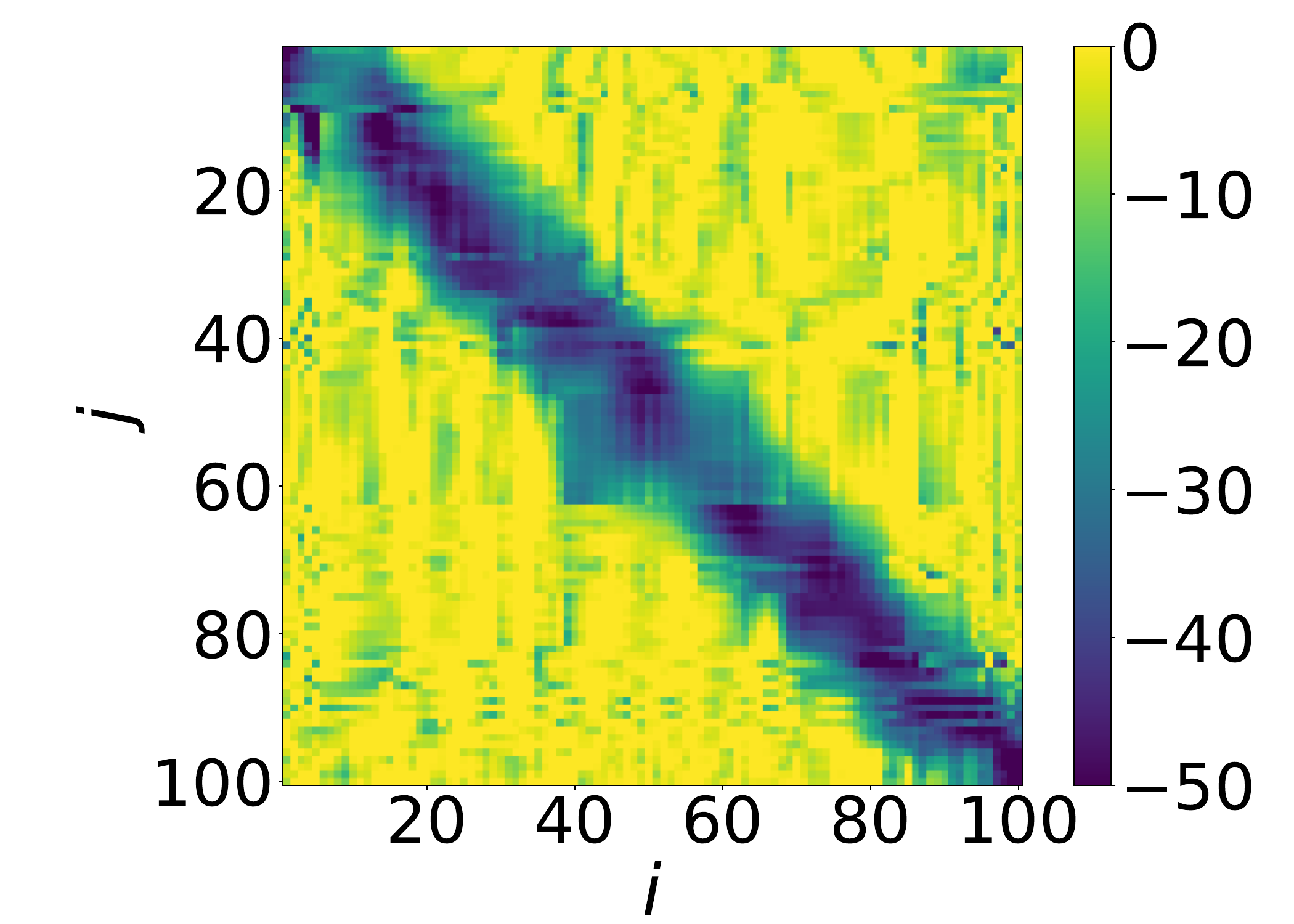}}\qquad
    \caption{The inverse solution for $n=100$ neurons with noise added to $\bb$ with a $1\%$ noise level and where the true connectivity is symmetric. The minimal Frobenius norm is used to calculate the regularization parameter, and the regularization parameter is $\kappa=22$.} 
    \label{fig:nx100_noiseb_symmetric}
\end{figure}

\begin{figure}
    \centering
    \subfigure[Noise level: $5\%$, Regularization parameter $\kappa=28$.]{\includegraphics[width=0.27\linewidth]{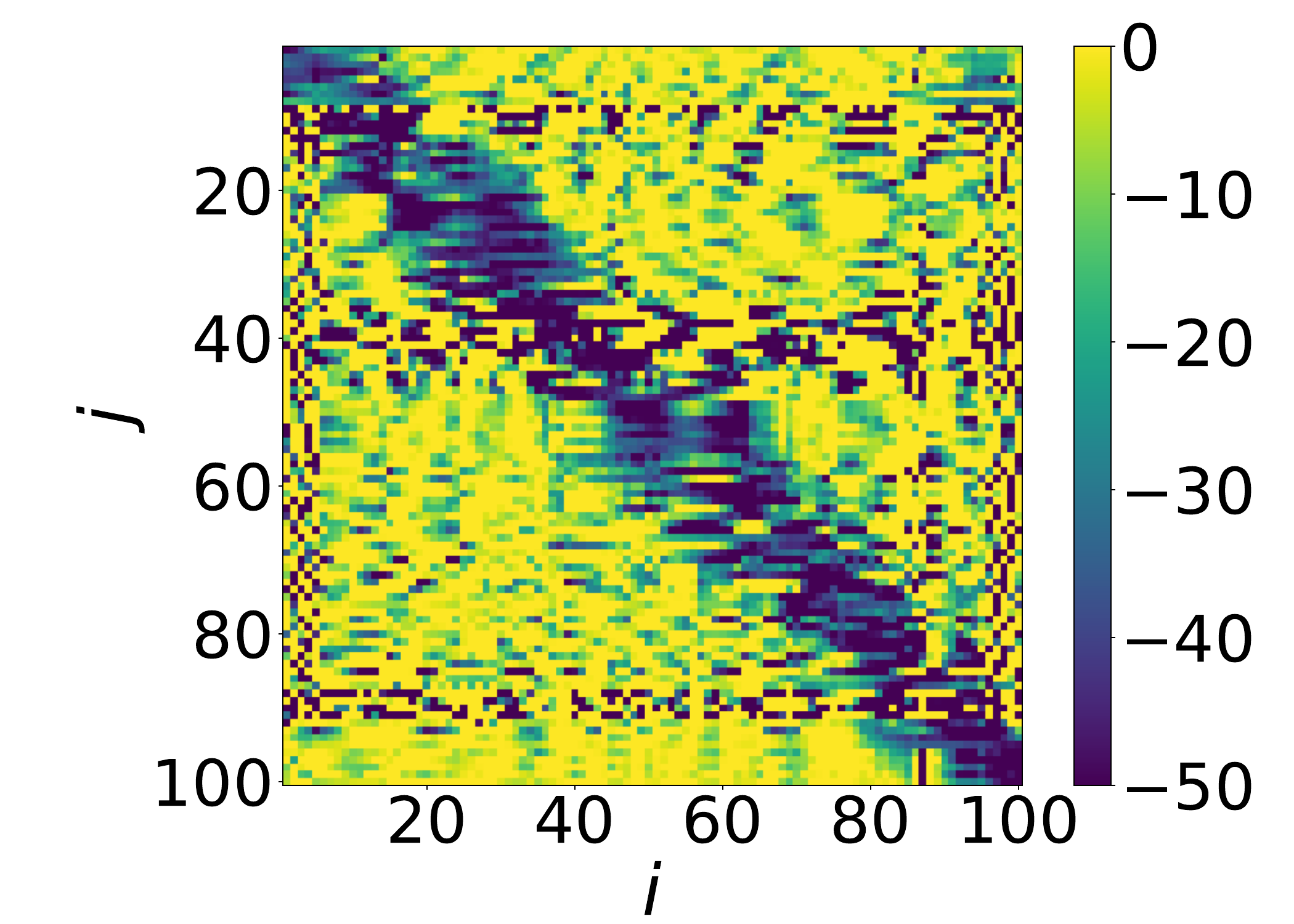}}\qquad
    \subfigure[Noise level: $10\%$, Regularization parameter $\kappa=16$.]
    {\includegraphics[width=0.27\linewidth]{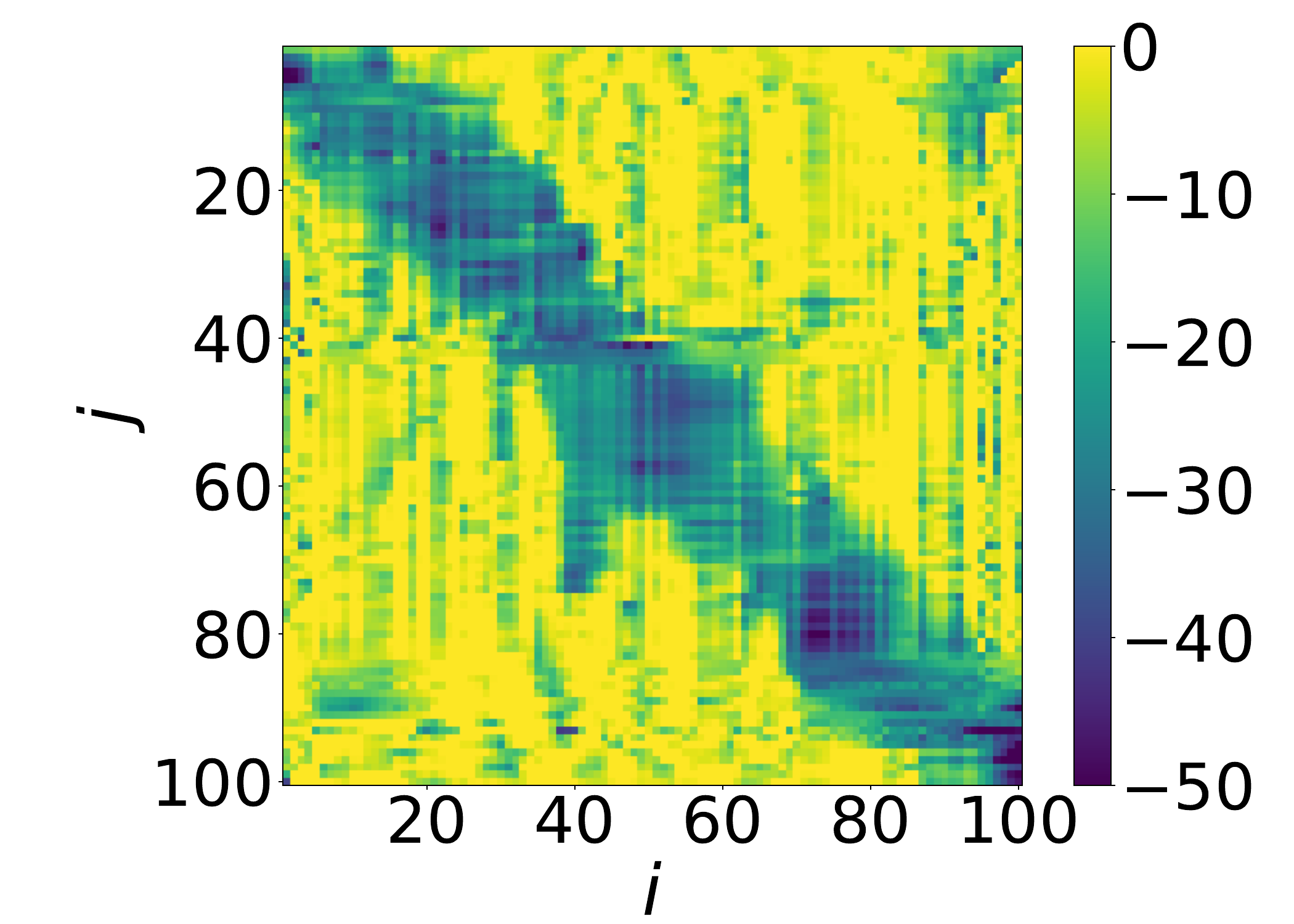}}\qquad
    \caption{The inverse solutions for $n=100$ neurons with noise added to $\bb$ and where the true connectivity is symmetric.}
    \label{fig:nx100_noiseb_symmetric_noise10}
\end{figure}

\begin{figure}
    \centering
    \subfigure[Noise level: $1\%$. Regularization parameter: $\kappa=16$.]{\includegraphics[width=0.27\linewidth]{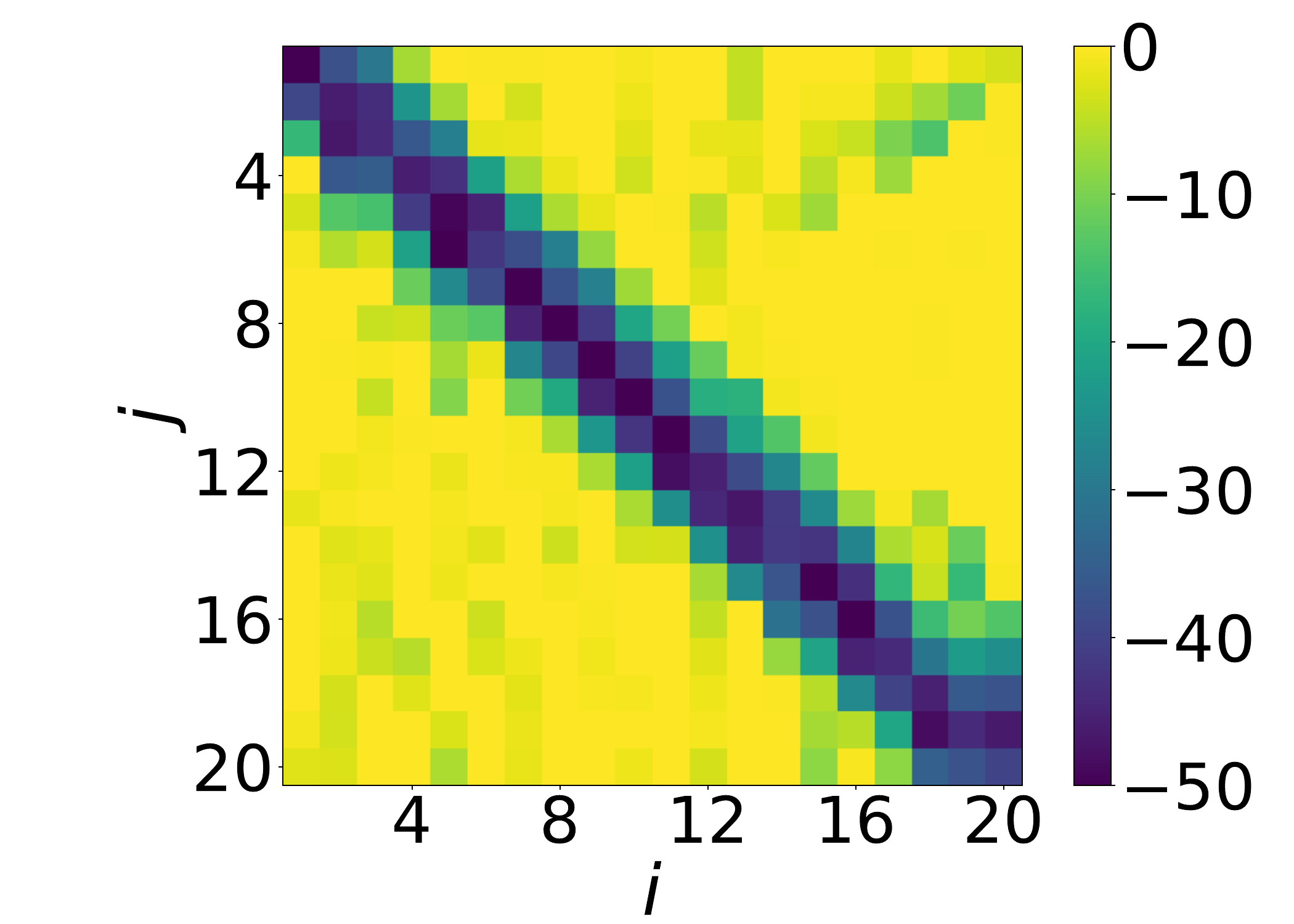}}\qquad
    \subfigure[Noise level: $5\%$. Regularization parameter: $\kappa=13$.]{\includegraphics[width=0.27\linewidth]{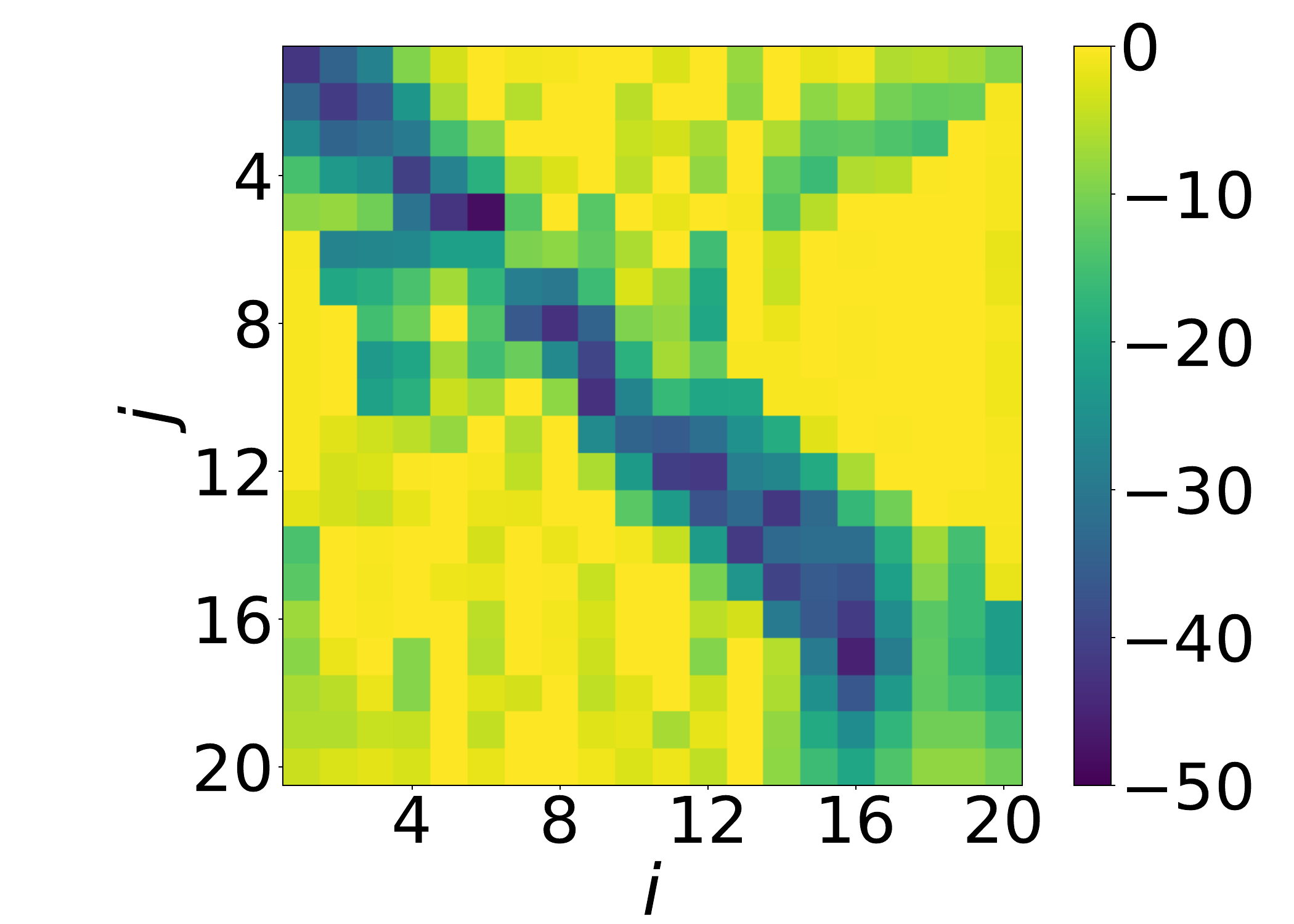}}\qquad
    \subfigure[Noise level: $10\%$. Regularization parameter: $\kappa=10$.]{\includegraphics[width=0.27\linewidth]{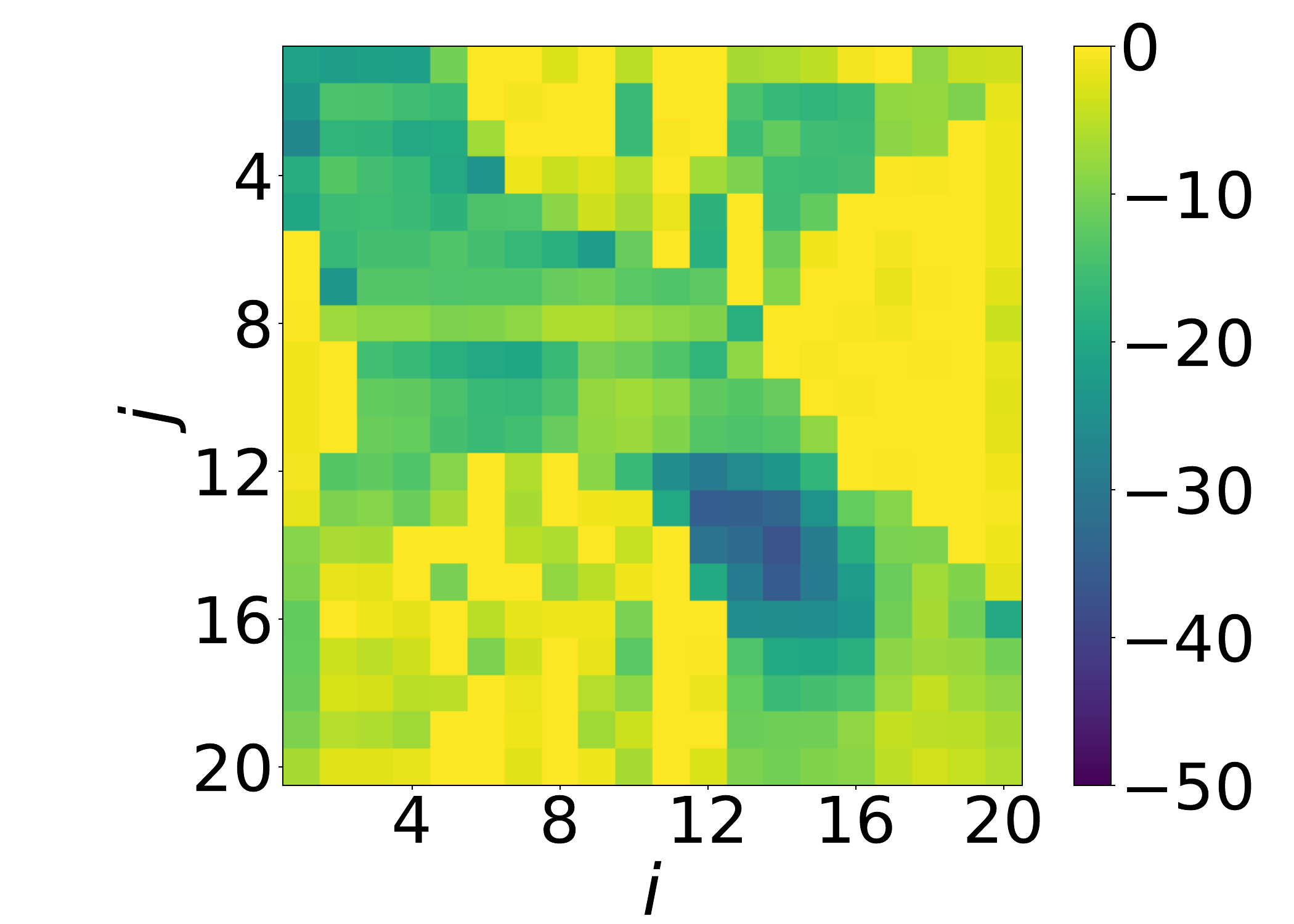}}\\
    \caption{The inverse solutions for $n=20$ neurons with noise added to the firing intervals and where the true connectivity function is symmetric.} 
    \label{fig:nx20_noisefire_symmetric}
\end{figure}

For $nl=5\%$ and $nl=10\%$ we get a reasonable value for $\kappa$ using Morozov's discrepancy principle, see Figure \ref{fig:nx100_noiseb_symmetric_noise10}.

When using Morozov's discrepancy principle, the stopping criterion is only dependent on the noise added to $\bb$, and the aforementioned discretization error occurring in $\AA$ is not taken into account. Thus, we suspect that for $nl=1\%$, the error in $\AA$ dominates the noise added to $\bb$ such that the stopping criterion in \eqref{eq:Dsicrep_b}  is not adequate to determine a reasonable $\kappa$. By reducing the time step used in the Euler method from $1/500$ to $1/5000$, the relative error between $\WW$ and $\WW^{\kappa}_{\text{inv}}$ measured in the Frobenius norm was reduced from $4.834$ to $0.658$ for $nl=5\%$.
This supports the validity of our suspicion.

Next, we consider the case where noise is added to the firing intervals. In figures \ref{fig:nx20_noisefire_symmetric} and \ref{fig:nx100_noisefire_symmetric} we see the results when we consider networks with $n=20$ and $n=100$ neurons, respectively. As expected, increased noise will result in less accurate results.

\begin{figure}
    \centering
    \subfigure[Noise level: $1\%$. Regularization parameter: $\kappa=61$.]{\includegraphics[width=0.27\linewidth]{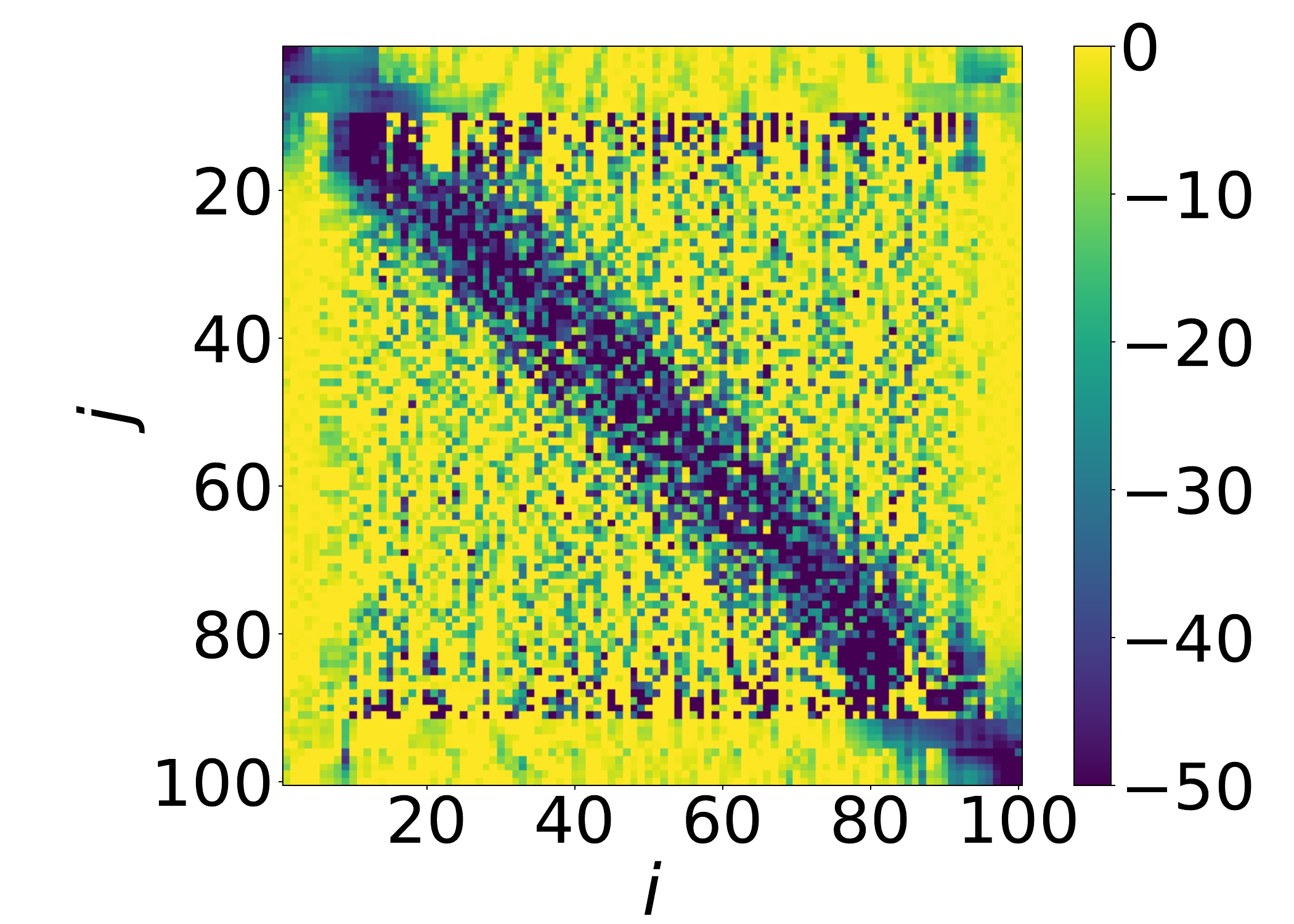}}\qquad
    \subfigure[Noise level: $5\%$. Regularization parameter: $\kappa=38$.]{\includegraphics[width=0.27\linewidth]{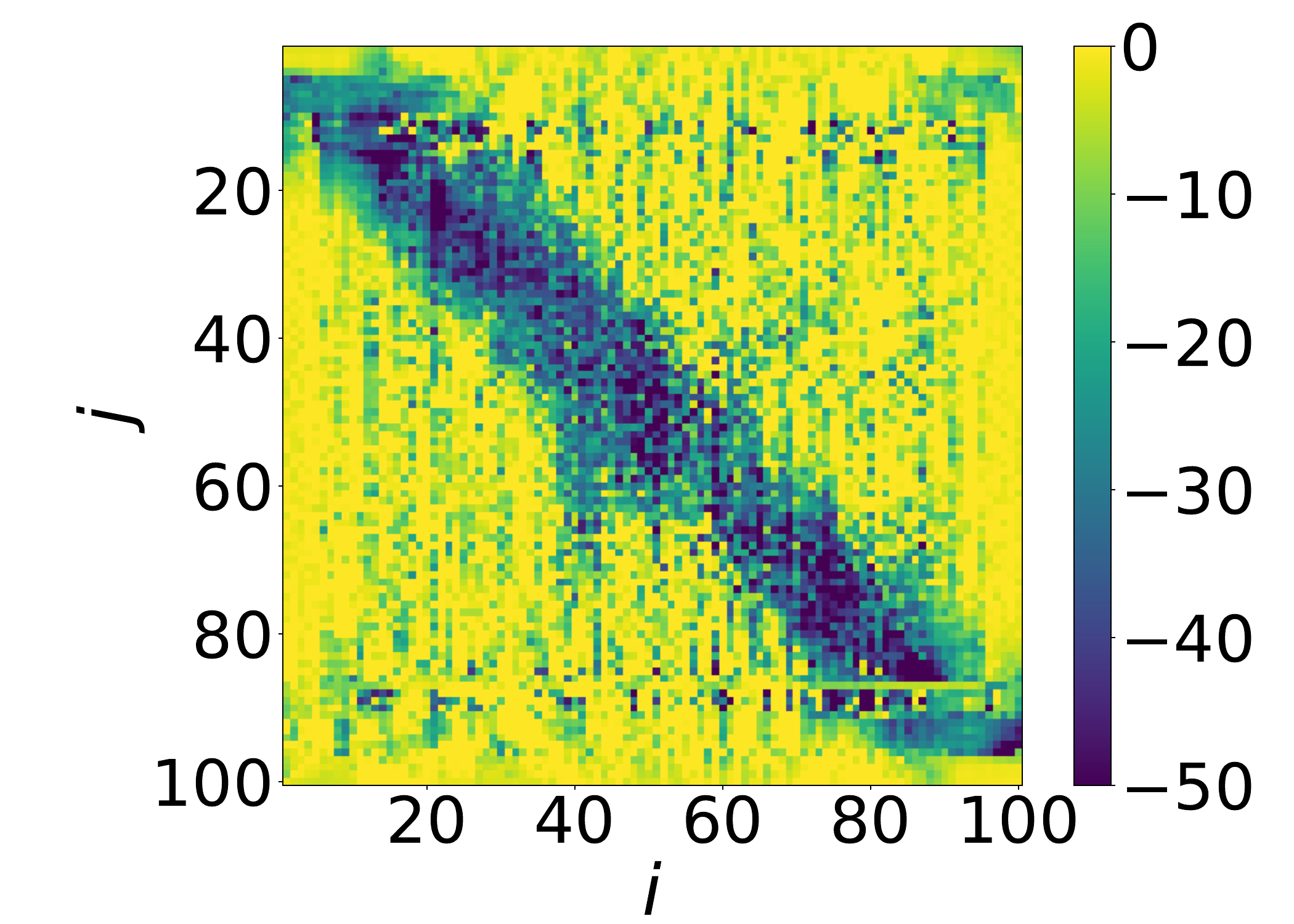}}\qquad
    \subfigure[Noise level: $10\%$. Regularization parameter: $\kappa=18$.]{\includegraphics[width=0.27\linewidth]{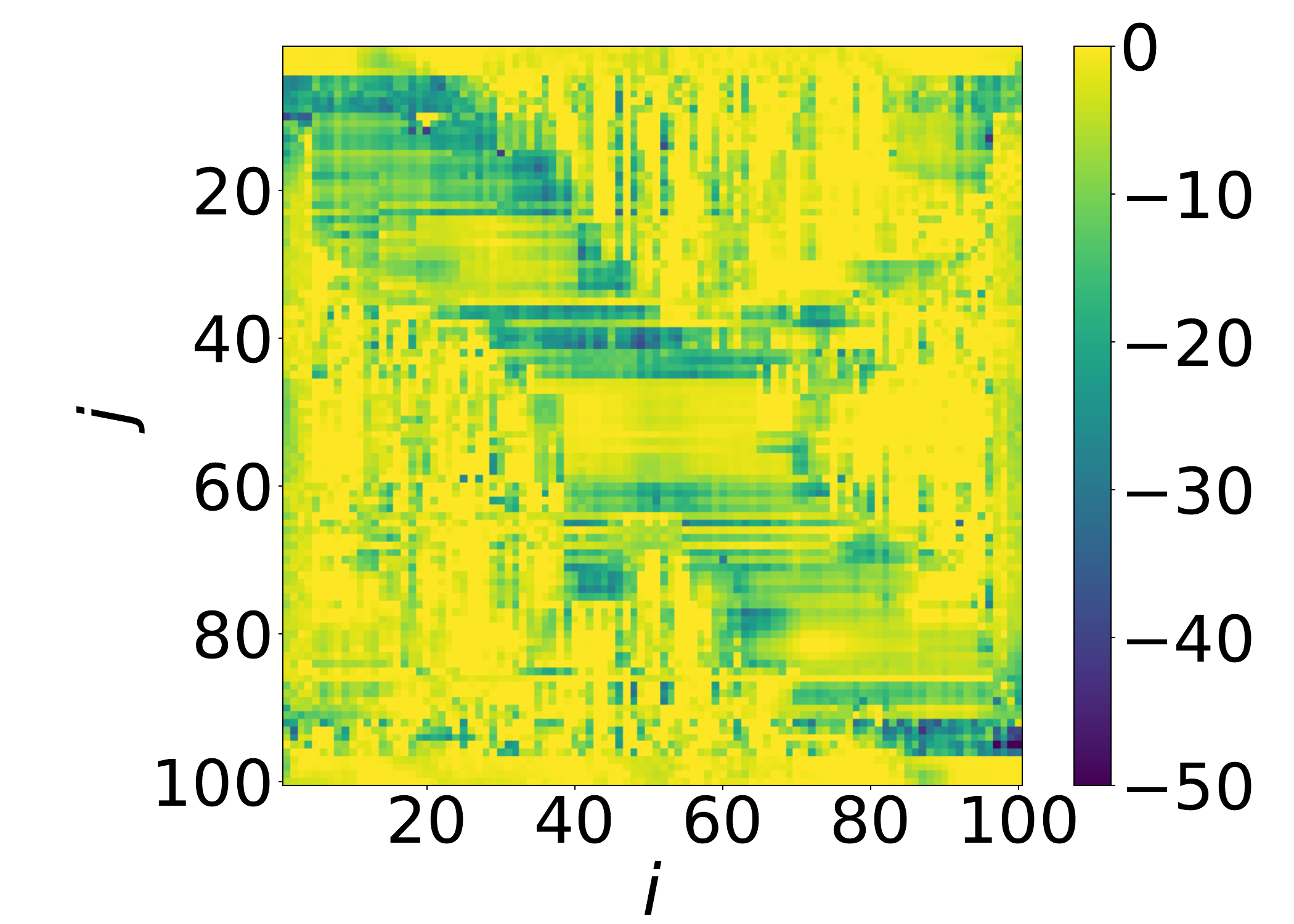}}\\
    \caption{The inverse solutions for $n=100$ neurons with noise added to the firing intervals and where the true connectivity function is symmetric.}
    \label{fig:nx100_noisefire_symmetric}
\end{figure}

In Table \ref{tab:sym} we present the relative error measured in the Frobenius norm for the reconstruction of the symmetric connectivity. For $n=20$ we observe that the relative error increases as the noise level increases for both the case where the noise is added to $\bb$ and the case where noise is added to the firing intervals. Note that for $n=100$ neurons and the $1\%$ noise level the regularization parameter is obtained by minimizing the Frobenius norm. Thus, this relative error is unrealistically small, and we cannot compare it to the other cases.

\begin{table}
    \centering
    \begin{tabular}{|c|c|c|c|c|c|c|}
    \hline
       & \multicolumn{3}{c|}{Noise added to $\bb$}  &  \multicolumn{3}{c|}{Noise added to $\{I_i^k\}$}\\
      \hline
        Noise level & $1\%$ & $5\%$ & $10\%$ & $1\%$ & $5\%$ & $10\%$\\
      \hline
        $n=20$ & $0.195$ & $0.515$ & $0.632$ & $0.209$ & $0.522$ & $0.741$ \\
         \hline
        $n=100$ &$0.382$* & $4.834$ & $0.589$&  $1.276$ & $0.713$ & $0.869$\\
        \hline
    \end{tabular}
    \caption{Relative error in the Frobenius norm when the true connectivity is symmetric, see Figure \ref{fig:W_sym}. The relative error for $n=100$, noise added to $\bb$ and $nl=1\%$ (marked with *), is not comparable to the rest: in this case the regularization parameter is chosen by minimizing the Frobenius norm.}
    \label{tab:sym}
\end{table}

\subsection{Non-symmetric connectivities}\label{sec:nonsym}
Here we consider the non-symmetric connectivity function \eqref{eq:Wnonsym}, see Figure \ref{fig:W_nonsym}. Again, we are going to both add noise to $\bb$ and to $\{I_i^k\}$. 

In Figure \ref{fig:nx20_noiseb_nonsymmetric} we see the results for the system with $n=20$ neurons, where noise is added to $\bb$. As in the symmetric case, using the eyeball norm we see that an increased amount of noise yields less accurate results. The results for $n=100$ neurons are reasonable for all the chosen noise levels, see Figure \ref{fig:nx100_noiseb_nonsymmetric}. In contrast to the symmetric case, the discrepancy principle successfully obtains a reasonable value for $\kappa$ for all three noise levels.

Lastly, we consider the case where the noise is added to the firing intervals $\{I_i^k\}$ for the non-symmetric connectivity. From Figures \ref{fig:nx20_noisefire_nonsymmetric} and \ref{fig:nx100_noisefire_nonsymmetric} we see that we get decent results for all three noise levels with both $n=20$ and $n=100$ neurons. 

Table \ref{tab:non_sym} shows the relative errors measured in the Frobenius norm for the non-symmetric connectivity matrix. In all four cases the relative error increases when we increase the noise level. However, as opposed to the result for the symmetric connectivity matrix, the relative error is smaller for $n=100$ neurons than for $n=20$ neurons. We suspect that the reason for the error decreasing is that for $n=100$ neurons we run the simulation for an increased amount of time. Thus, obtaining more information to determine the connectivity function. However, this is not the case for the symmetric connectivity. This is likely due to the fact that the singular values of $\AA$ for the symmetric connectivity tend to zero much faster than for the non-symmetric connectivity, which can be seen in Figure \ref{fig:m_vs_s}. Thus, a larger error amplification is expected in the symmetric case.
\begin{figure}
    \centering
    \subfigure[Noise level: $1\%$. Regularization parameter: $\kappa=11$.]{\includegraphics[width=0.27\linewidth]{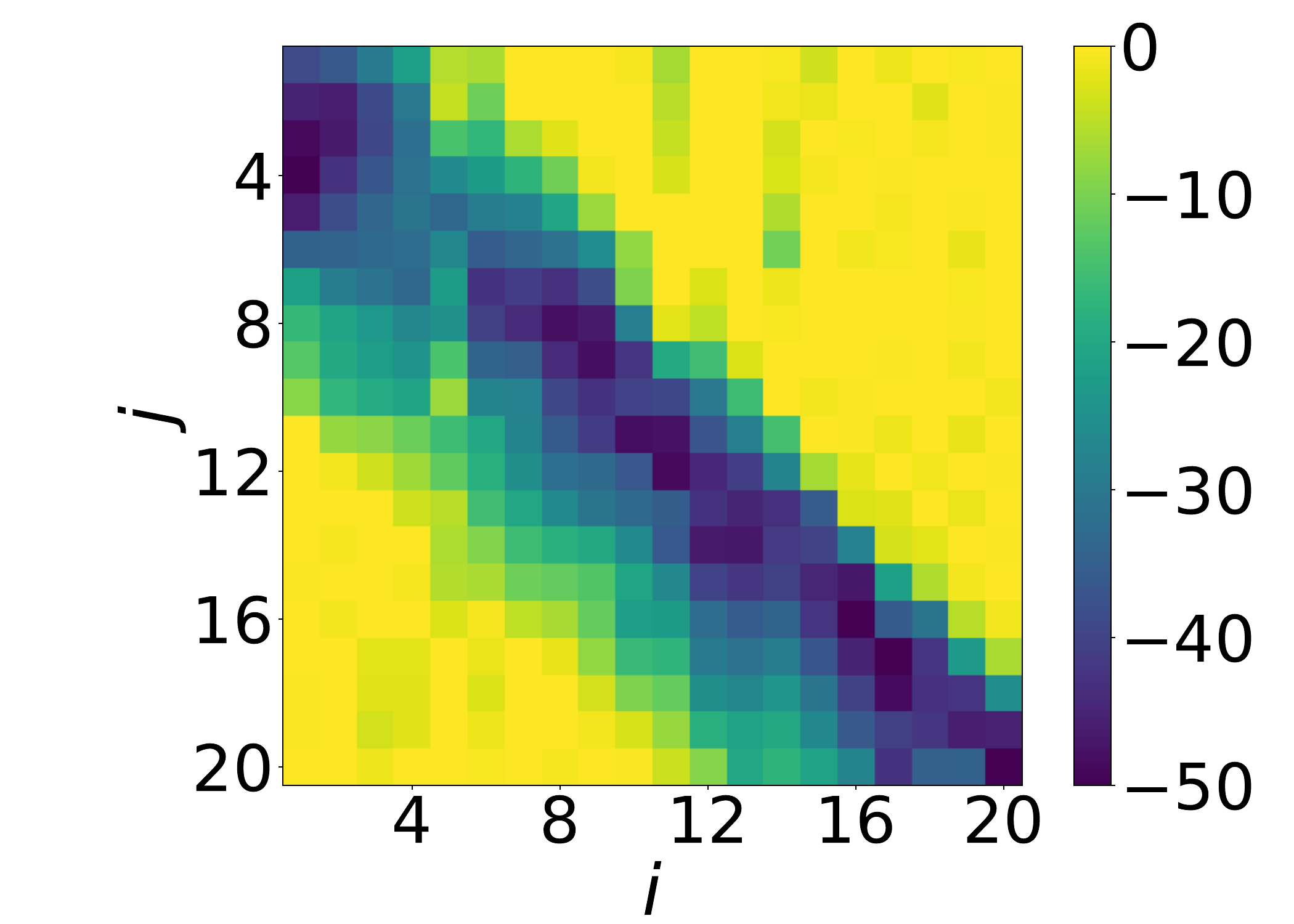}}\qquad
    \subfigure[Noise level $5\%$. Regularization parameter: $\kappa=9$.]{\includegraphics[width=0.27\linewidth]{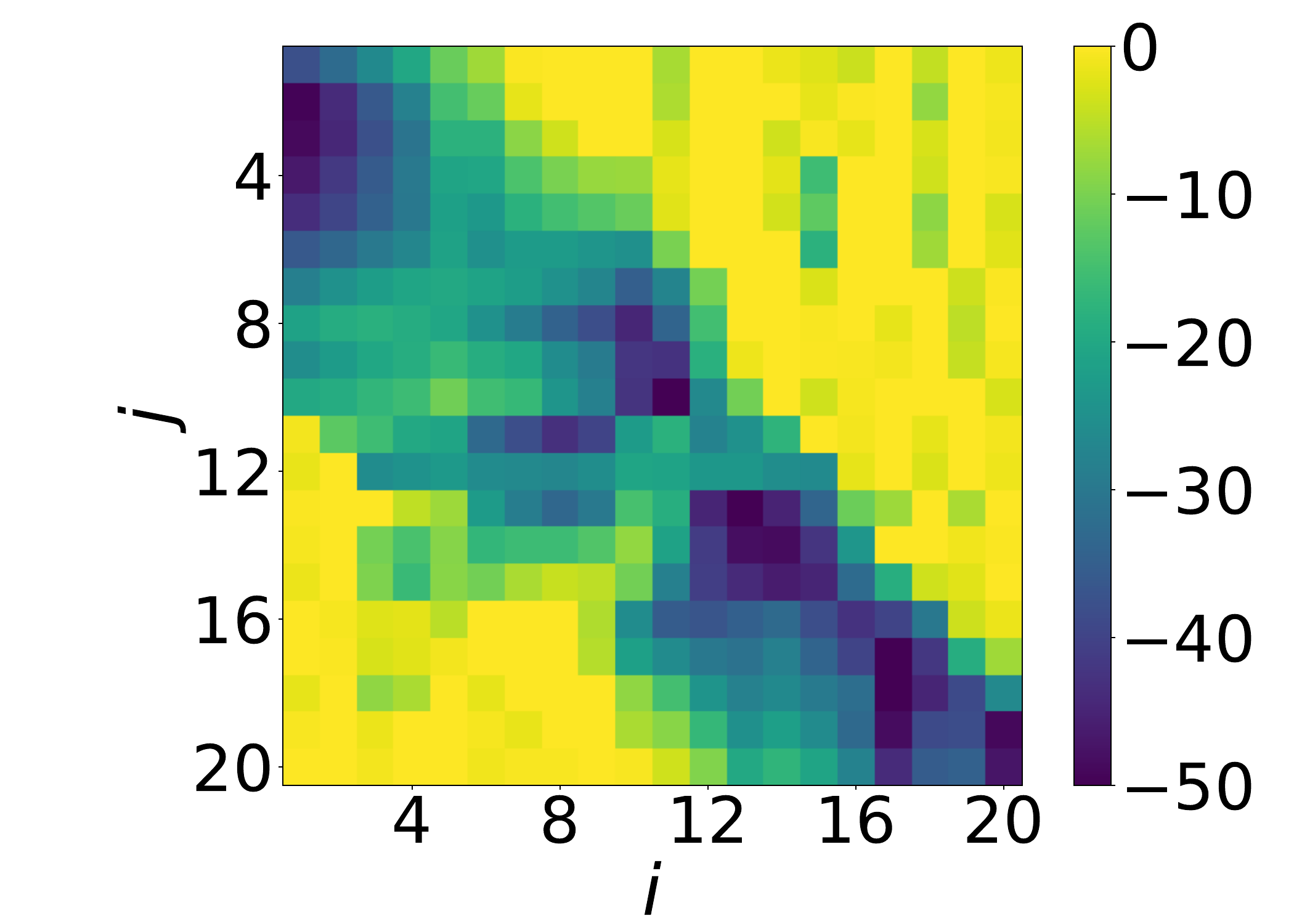}}\qquad
    \subfigure[Noise level: $10\%$. Regularization parameter: $\kappa=8$.]{\includegraphics[width=0.27\linewidth]{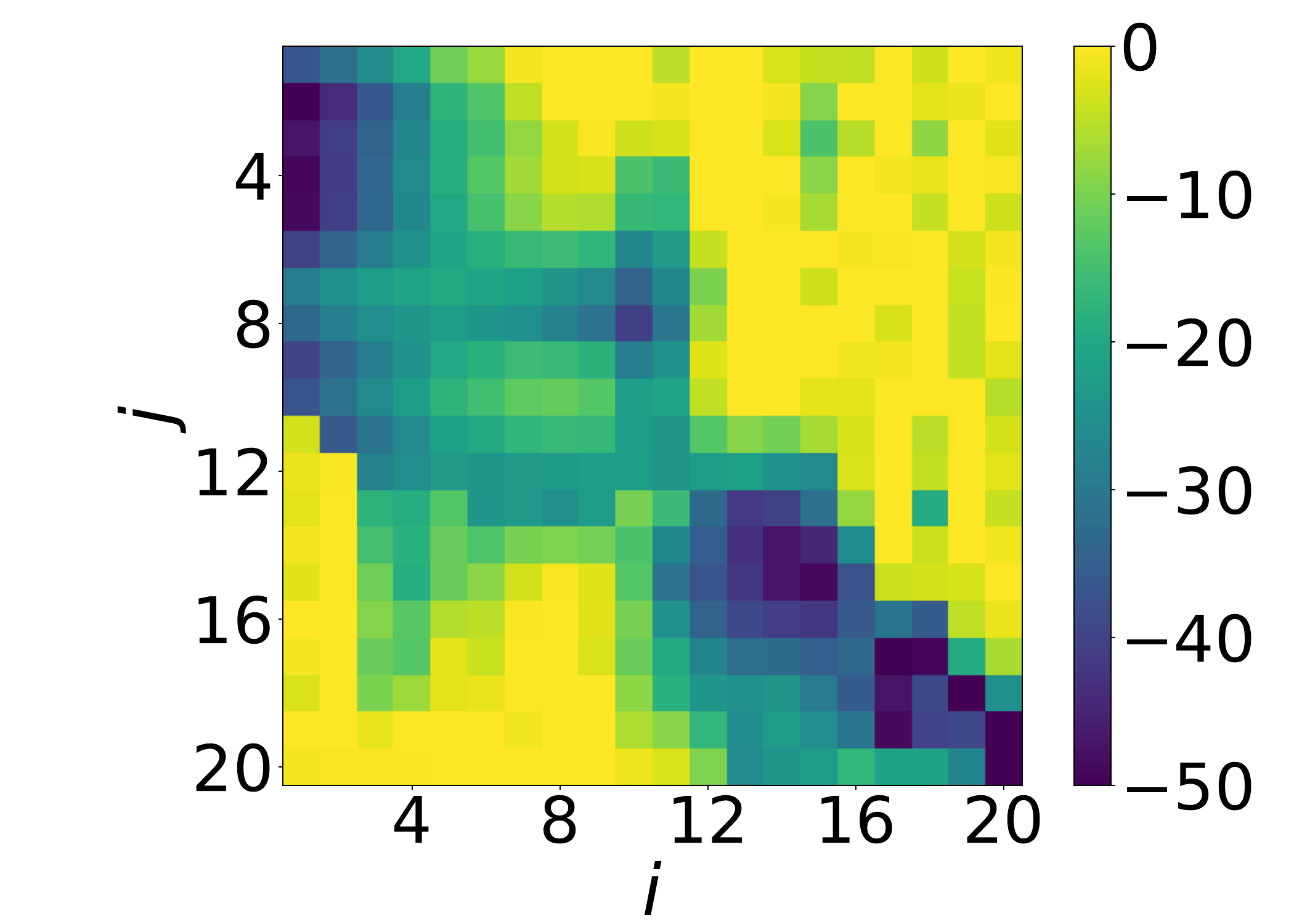}}\\
    \caption{The inverse solutions for $n=20$ neurons with noise added to $\bb$ and where the true connectivity function is non-symmetric.}
    \label{fig:nx20_noiseb_nonsymmetric}
\end{figure}

\begin{figure}
    \centering
    \subfigure[Noise level: $1\%$. Regularization parameter: $\kappa=27$.]{\includegraphics[width=0.27\linewidth]{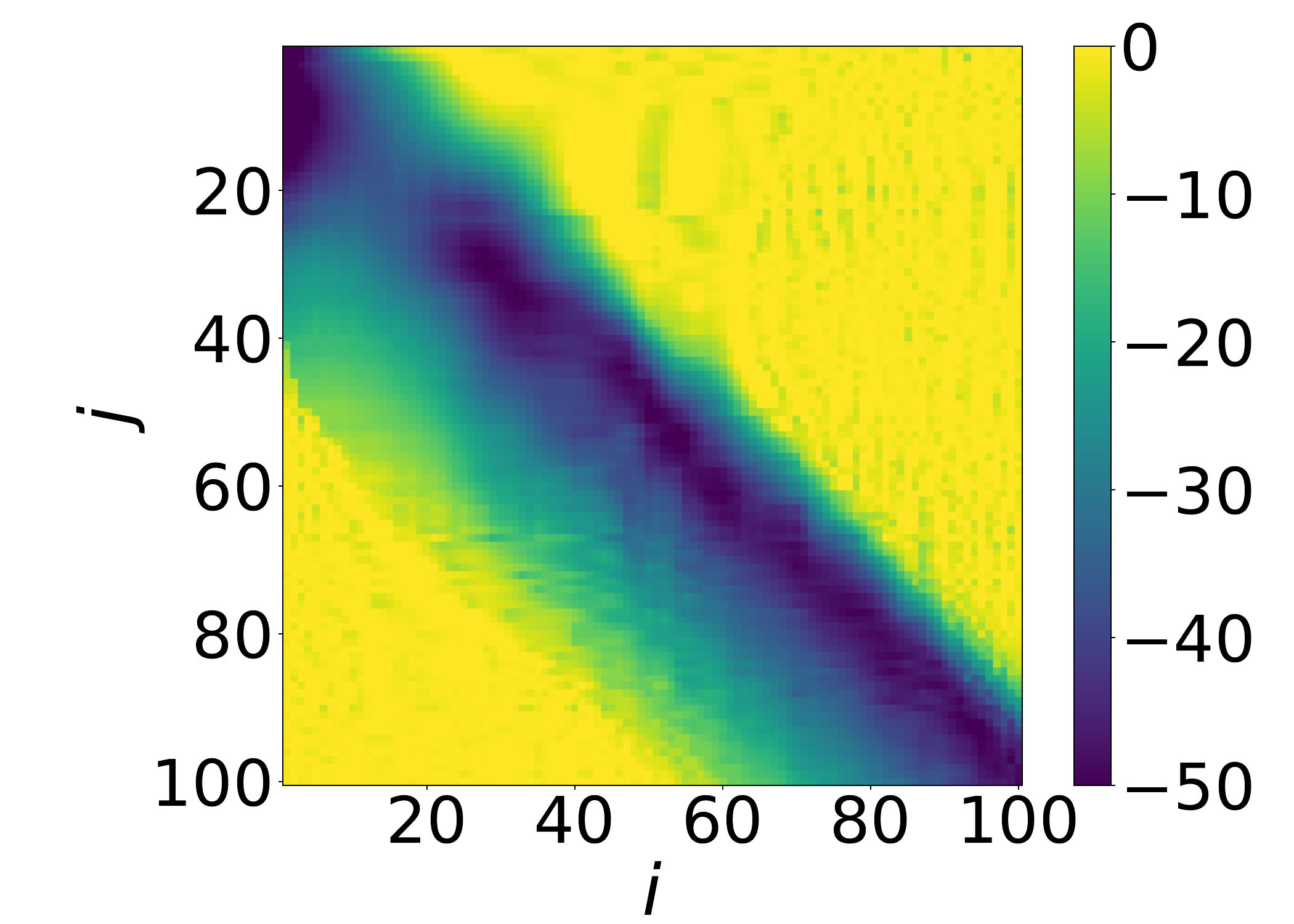}}\qquad
    \subfigure[Noise level: $5\%$. Regularization parameter: $\kappa=19$.]{\includegraphics[width=0.27\linewidth]{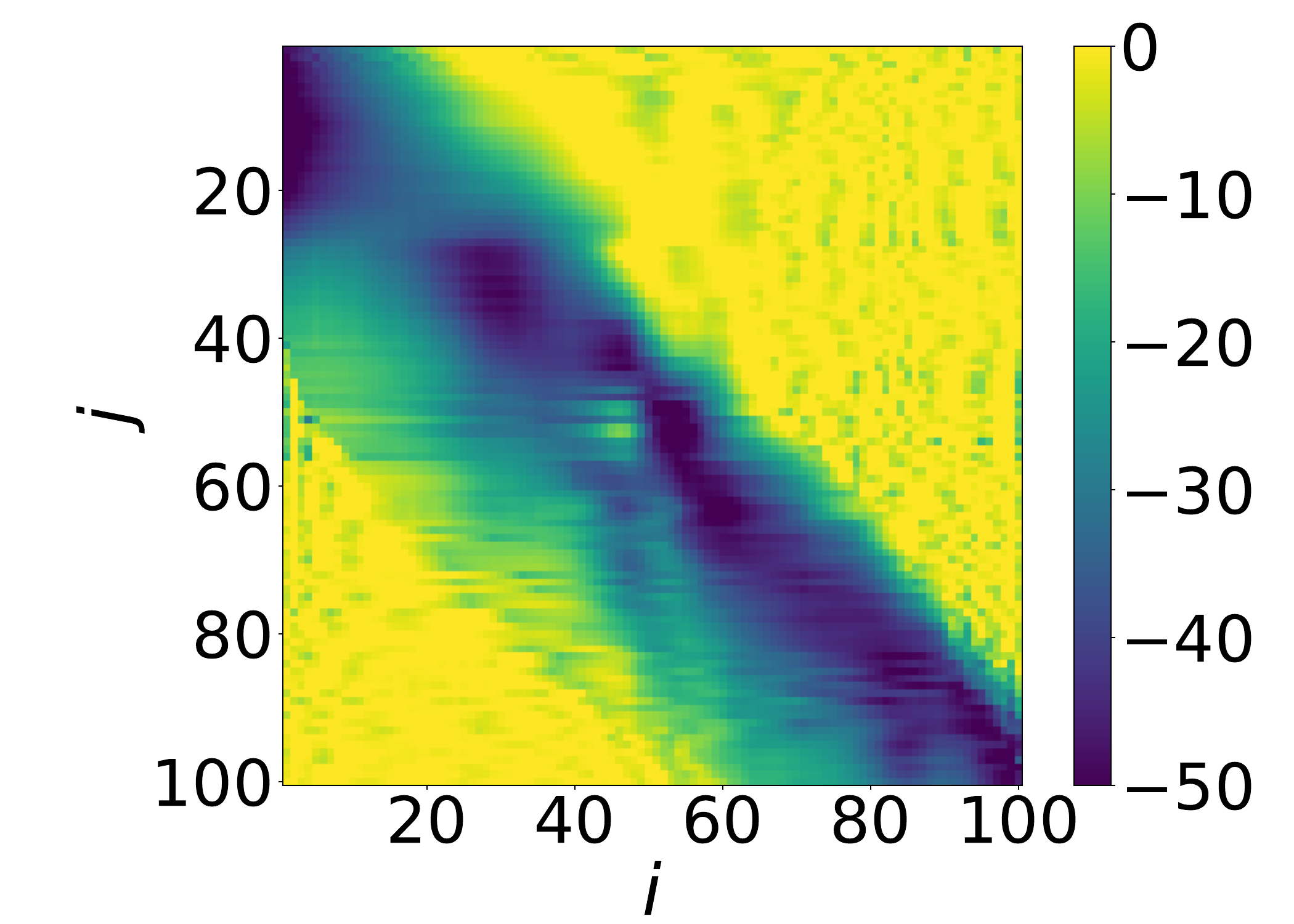}}\qquad
    \subfigure[Noise level: $10\%$. Regularization parameter: $\kappa=17$.]{\includegraphics[width=0.27\linewidth]{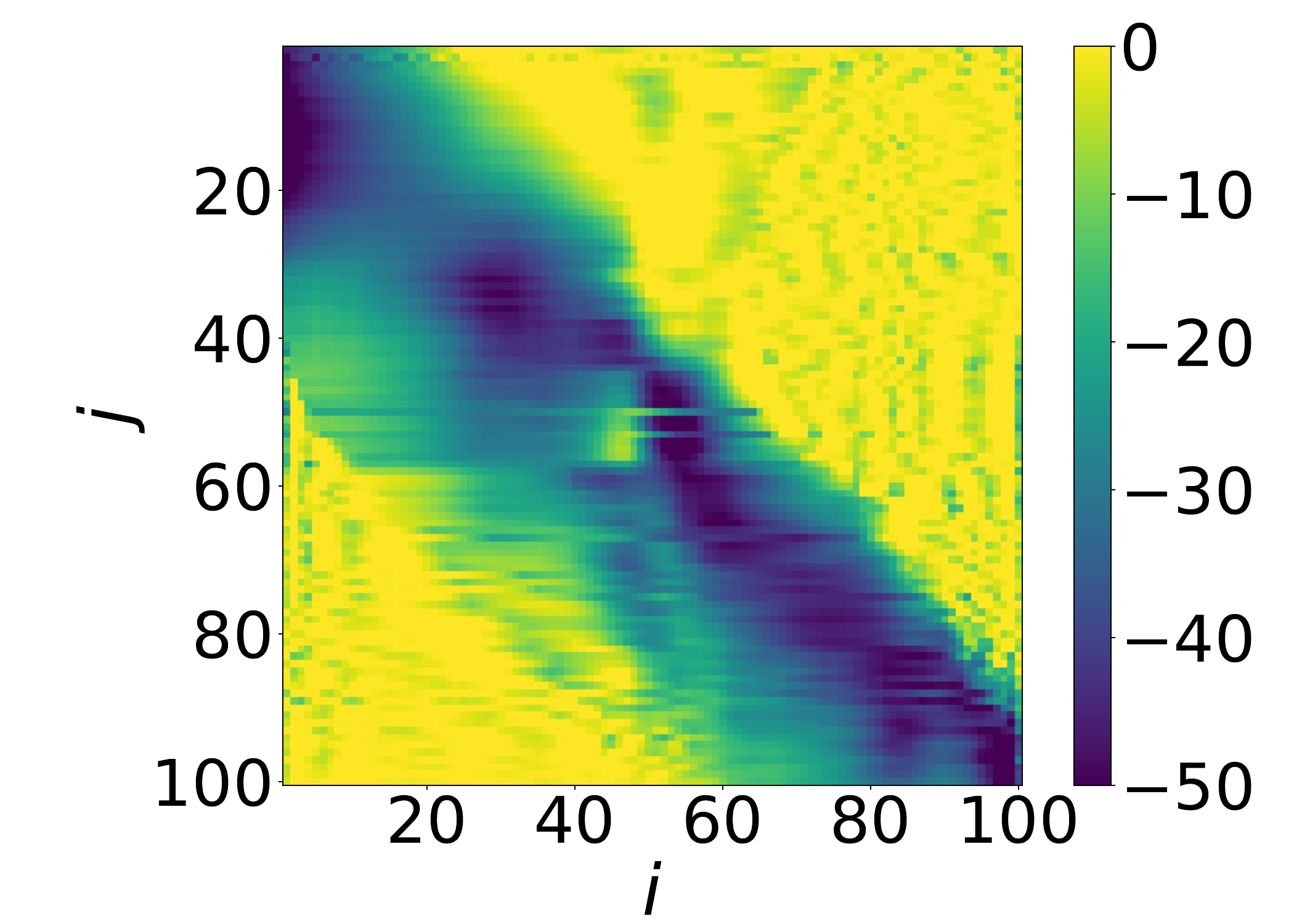}}\\
    \caption{The inverse solutions for $n=100$ neurons with the noise added to $\bb$ and where the true connectivity function is non-symmetric.}
    \label{fig:nx100_noiseb_nonsymmetric}
\end{figure}

\begin{figure}
    \centering
    \subfigure[Noise level: $1\%$. Regularization parameter: $\kappa=11$.]{\includegraphics[width=0.27\linewidth]{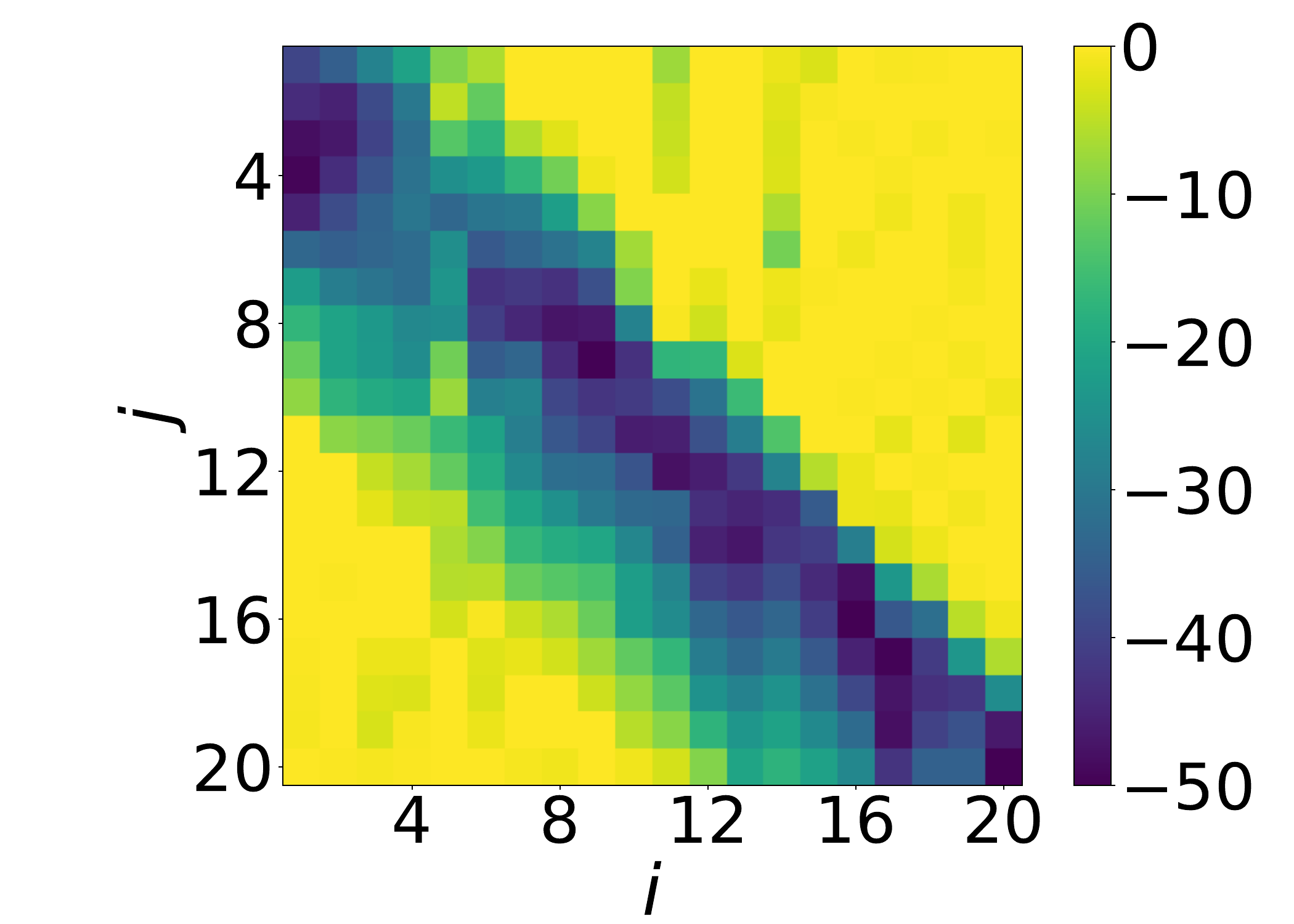}}\qquad
    \subfigure[Noise level: $5\%$. Regularization parameter: $\kappa=13$.]{\includegraphics[width=0.27\linewidth]{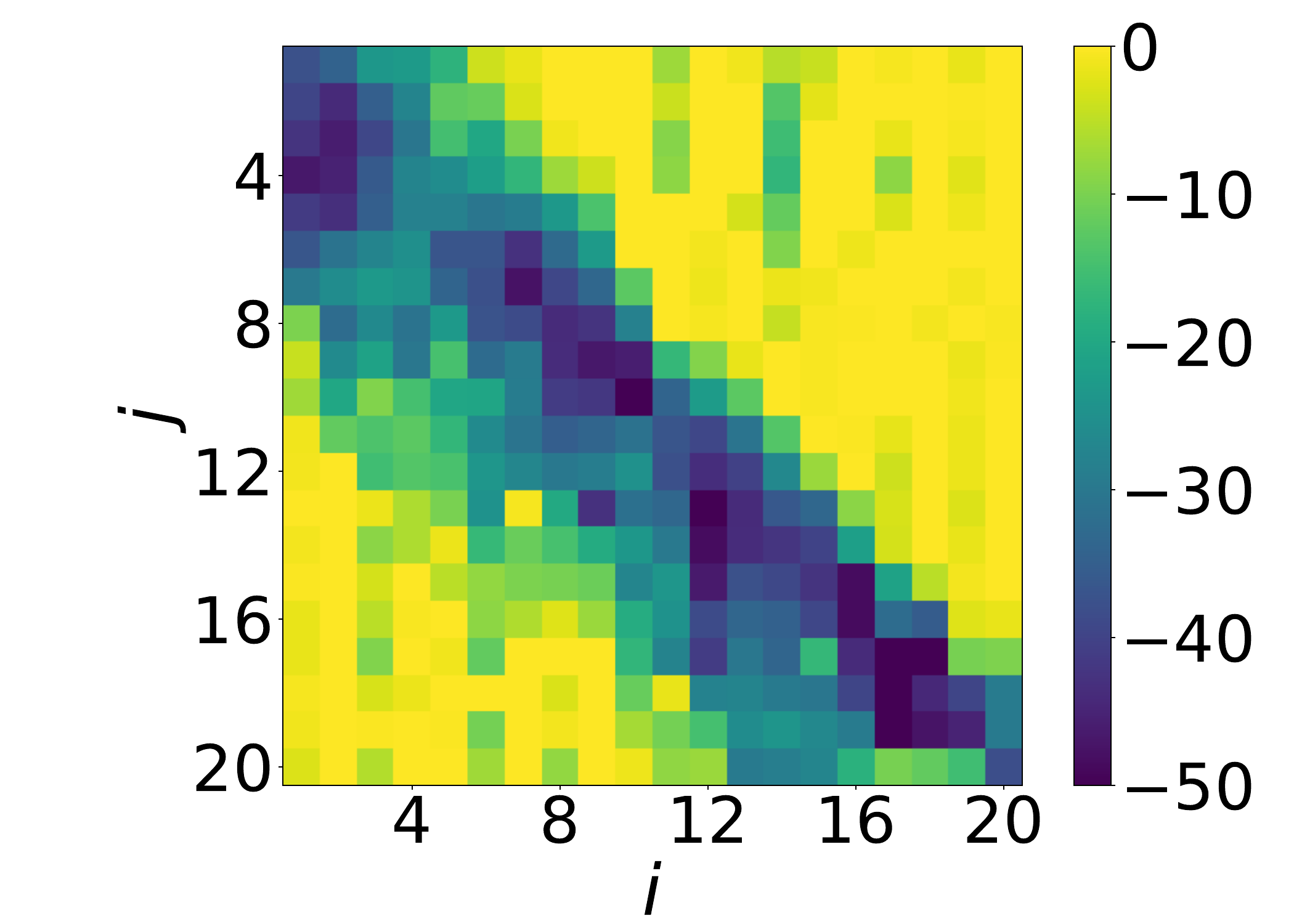}}\qquad
    \subfigure[Noise level: $10\%$. Regularization parameter: $\kappa=7$.]{\includegraphics[width=0.27\linewidth]{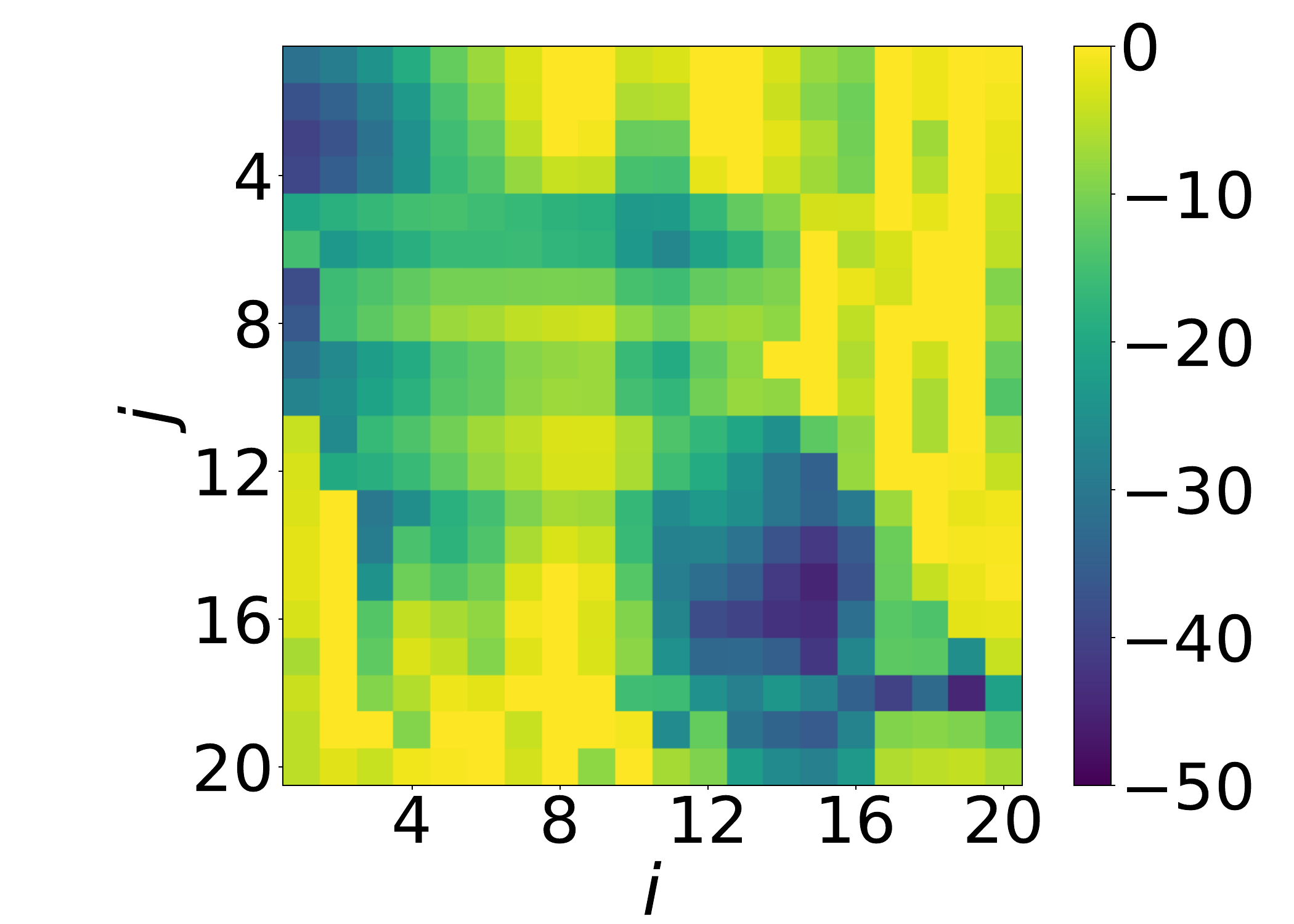}}\\
    \caption{The inverse solutions for $n=20$ neurons with the noise added to the firing intervals and where the true connectivity function is non-symmetric.}
    \label{fig:nx20_noisefire_nonsymmetric}
\end{figure}

\begin{figure}
    \centering
    \subfigure[Noise level: $1\%$. Regularization parameter: $\kappa=35$.]{\includegraphics[width=0.27\linewidth]{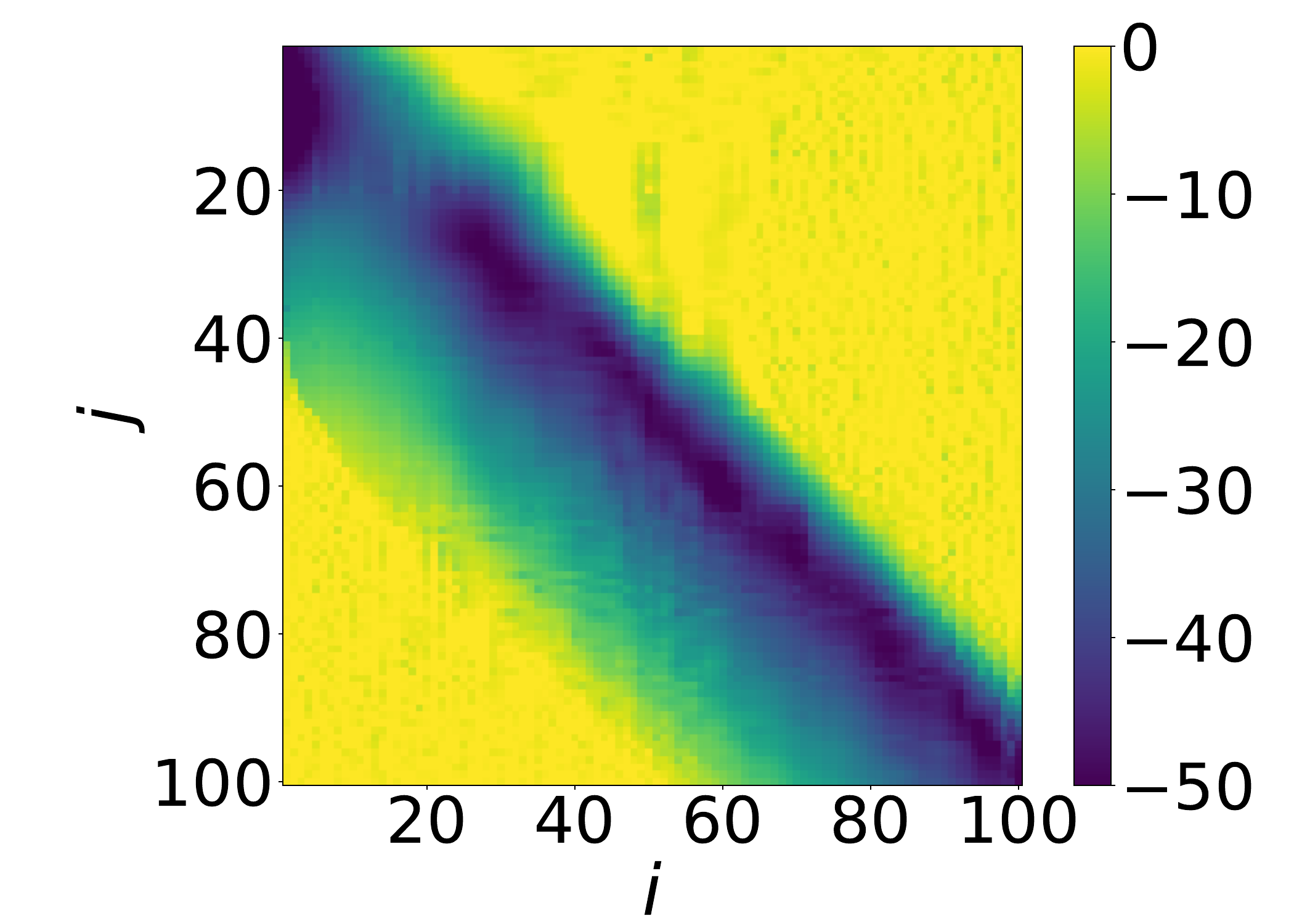}}\qquad
    \subfigure[Noise level: $5\%$. Regularization parameter: $\kappa=31$.]{\includegraphics[width=0.27\linewidth]{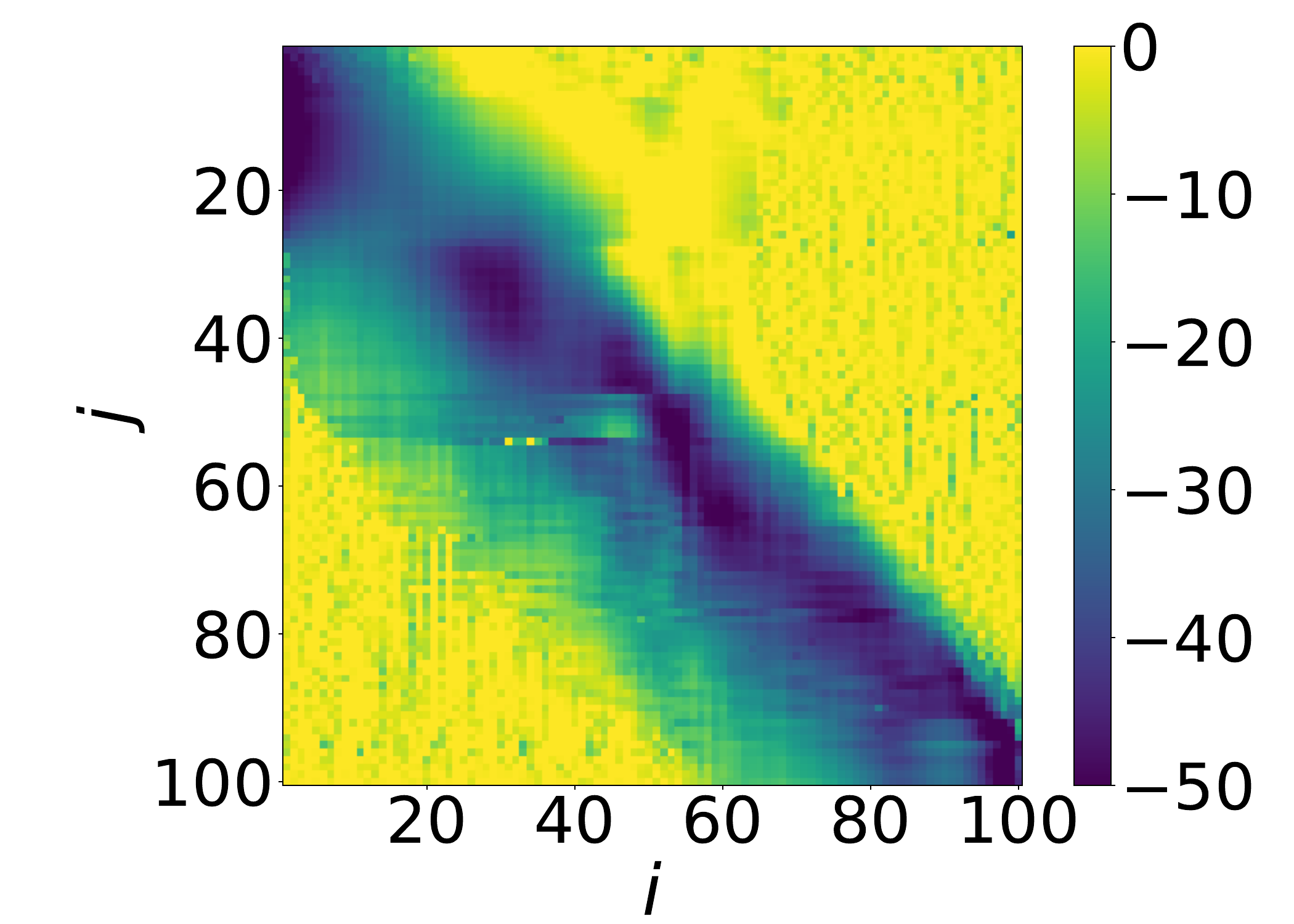}}\qquad
    \subfigure[Noise level: $10\%$. Regularization parameter: $\kappa=40$.]{\includegraphics[width=0.27\linewidth]{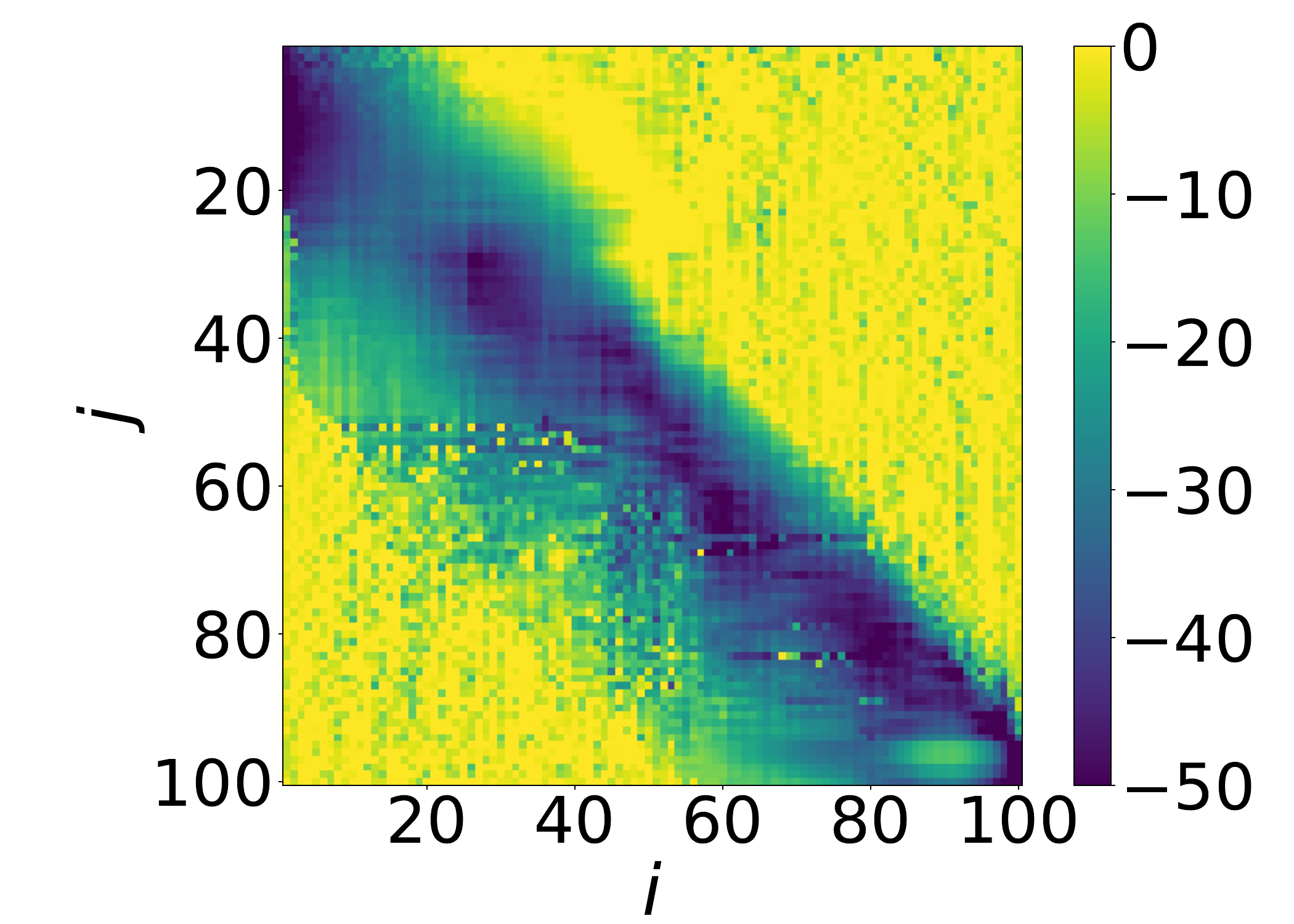}}\\
    \caption{The inverse solutions for $n=100$ neurons with the noise added to the firing intervals and where the true connectivity function is non-symmetric.} 
    \label{fig:nx100_noisefire_nonsymmetric}
\end{figure}

\begin{table}
    \centering
    \begin{tabular}{|c|c|c|c|c|c|c|}
    \hline
       & \multicolumn{3}{c|}{Noise added to $\bb$}  &  \multicolumn{3}{c|}{Noise added to $\{I_i^k\}$}\\
      \hline
        Noise level & $1\%$ & $5\%$ & $10\%$ & $1\%$ & $5\%$ & $10\%$\\
      \hline
        $n=20$ & $0.213$ & $0.393$ & $0.484$ &  $0.218$ & $0.307$ & $0.651$\\
         \hline
        $n=100$ &$0.129$ & $0.211$ & $0.259$&  $0.119$ & $0.211$ & $0.274$\\
        \hline
    \end{tabular}
    \caption{The relative error measured in the Frobenius norm when the true connectivity is non-symmetric, see Figure \ref{fig:W_nonsym}.}
    \label{tab:non_sym}
\end{table}

\section{Concluding remarks}\label{sec:concluding}
The purpose of this work was to test the waters to see if it is possible to predict the effective connection strengths in a network of neurons using a simple neural field model. A next step would be to loosen the somewhat restrictive assumptions that all other parameters in the problem are given. For example, when considering real data, we may not know the initial data, $\{ s_i^0 \}$. If that is the case, the estimation of $\{ s_i^0 \}$ has to be included in the inverse problem. However, the resulting system of equations will then no longer be linear, which requires its own investigation.

The value of the time delay $\tau_d$ is of particular importance, as it significantly influences the network dynamics \cite{roxin2005role}. Furthermore, from \eqref{eq:baby_linear_system}, we observe that $\tau_d$ plays a crucial role in the inverse problem. Therefore, a thorough investigation of the impact of the time delay on the system’s behaviour is warranted. However, this comprehensive analysis is beyond the scope of this article.

Another challenge that arises when considering real data instead of synthetic data, is that there will typically be more neurons in the network than what we have considered. This will lead to an increased run time, and more firing data will be needed to prevent having an under-determined system. On the other hand, if the collection of firing data includes more firing events than the number of neurons, we have an over-determined system. In this case, it can be beneficial to not use all the firing events to construct the inverse problem, as close firing events can cause the stability of the parameter estimation problem to deteriorate. Thus, an interesting question is which firing events we should include in the linear system to obtain a problem which is as stable as possible. This is a challenging optimization problem. 

Finally, for the approach described here to bear more merit, an in-depth analysis of the error propagation in the reconstruction process, based on the considerations in Section \ref{subsec:error}, needs to be performed. As this requires significant effort, such an investigation is left for a future work.
\\

\noindent {\small {\bf Data availability}. Data sharing not applicable to this article.} 

\noindent {\small {\bf Conflict of interest}. The authors declare no conflict of interest.}

\bibliographystyle{abbrv}
\bibliography{biblio.bib}

\end{document}